\documentclass[12pt, a4paper]{article}
\usepackage{amsfonts}
\usepackage{mathrsfs}
\usepackage{latexsym}
\usepackage{longtable}
\usepackage{hyperref}
\usepackage{xy}
\usepackage{amsfonts,amsmath,amssymb,amsthm}
\usepackage{color}
\usepackage{bitset}

\xyoption{all}

\newcommand{\bcen}{\begin{center}}     \newcommand{\ecen}{\end{center}}
\newcommand{\bay}{\begin{array}}      \newcommand{\eay}{\end{array}}
\newcommand{\beq}{\begin{eqnarray*}}      \newcommand{\eeq}{\end{eqnarray*}}

\def\Hom{\mathrm{Hom}}
\def\Ker{\mathrm{Ker}}
\def\Im{\mathrm{Im}}
\def\Ext{\mathrm{Ext}}
\def\Tor{\mathrm{Tor}}

\def\Mod{\mathrm{Mod}}

\def\lim{\mathrm{lim}}

\def\Coker{\mathrm{Coker}}

\begin{document}
	
	\newtheorem{theorem}{Theorem}[section]
	\newtheorem{proposition}[theorem]{Proposition}
	\newtheorem{lemma}[theorem]{Lemma}
	\newtheorem{corollary}[theorem]{Corollary}
	\newtheorem{remark}[theorem]{Remark}
	\newtheorem{example}[theorem]{Example}
	\newtheorem{definition}[theorem]{Definition}
	\newtheorem{question}[theorem]{Question}
	\numberwithin{equation}{section}

	\title{\large\bf
 Gluing of cotorsion pairs via recollements
of abelian categories}
	
	\author{\large Jinrui Yang, Yongyun Qin }
	\date{\footnotesize School of Mathematics, Yunnan Key Laboratory of Modern Analytical Mathematics and Applications,
Yunnan Normal University, Kunming, Yunnan 650500, China.
\\ E-mail: qinyongyun2006@126.com
	}
	
	\maketitle
	
	\begin{abstract} Let $(  \mathcal{A^{'}},\mathcal{A},\mathcal{A^{''}},i^\ast,i_\ast,i^!,j_!,j^\ast,j_\ast)$
be a recollement of abelian categories. Suppose
that we are given two cotorsion pairs $({\mathcal{U^{'}}},\mathcal{V{'}})$ and $({\mathcal{U}^{''}},{\mathcal{V}^{''}})$ in $\mathcal{A}^{'}$ and $\mathcal{A}^{''}$, respectively. We construct two cotorsion pairs
  $(^{\bot}{\mathcal{N}_{\mathcal{V^{''}}}^{\mathcal{V^{'}}}},
\mathcal{N}_{\mathcal{V^{''}}}^{\mathcal{V^{'}}})$ and
$(\mathcal{M}_{\mathcal{U^{''}}}^{\mathcal{U^{'}}}, ({\mathcal{M}_{\mathcal{U^{''}}}^{\mathcal{U^{'}}}})^\bot)$ in $\mathcal{A}$.
Moreover, we provide a sufficient condition for these two cotorsion pairs to coincide, and we
investigate the heredity and completeness of $(\mathcal{M}_{\mathcal{U^{''}}}^{\mathcal{U^{'}}}, \mathcal{N}_{\mathcal{V^{''}}}^{\mathcal{V^{'}}})$.
These results are applied to construct new cotorsion pairs in Morita rings. In the course of proof, we introduce a specific constraint on recollements of abelian categories, requiring
$\varepsilon_P$ to be a monomorphism for any projective $P \in \mathcal{A}$, with $\varepsilon: j_!j^* \to \mathrm{id}_{\mathcal{A}}$ being the counit of $(j_!, j^*)$. Such recollements enjoy rich homological properties and hence might be of independent
interest.
\end{abstract}
	\medskip
	
	
	\medskip
	
	
	\bigskip
	\section{\large Introduction}
	
	\indent\indent
	The notion of cotorsion pairs was introduced by Salce \cite{SL} as an analog of the classical torsion pairs
where the orthogonality of $\Hom$ is replaced by $\Ext ^1$.
Cotorsion pairs provide a perspective on resolution within a given category, and they are intricately
connected to other homological concepts of significant interest in category theory and representation
theory. Recently, the concept of cotorsion pairs has received a significant boost, largely due to Hovey's correspondence between abelian model structures and certain cotorsion pairs in abelian categories \cite{GJ, HM}.
Therefore, it is crucial to construct cotorsion pairs
in relative homological algebra and model structures \cite{EEE, GR, HM}.

Recollements have been introduced by Beilinson-Bernstein-Deligne in the context of triangulated categories \cite{BBD82},
and have been studied in abelian categories by many authors \cite{FV, GKP21, Psa14, PV14}.
It should be noted that recollements
of abelian categories appear more naturally than those of triangulated categories,
because any idempotent element of a ring induces a recollement situation
between module categories. Moreover, recollements of abelian categories
are very useful in gluing and in reduction; see for example
\cite{AKL11, Chen12, LVY14, MX2, YQ26, ZC17, Z17}. In particular,
many experts investigated when (complete hereditary) cotorsion
pairs in $\mathcal{A'}$ and $\mathcal{A''}$ can induce (complete hereditary) cotorsion pairs in
$\mathcal{A}$, where $(\mathcal{A^{'}}, \mathcal{A}, \mathcal{A^{''}}, i^\ast, i_\ast, i^!, j_!, j^\ast, j_\ast)$ is
a recollement of abelian (exact, triangulated or extriangulated) categories
\cite{Chen13, FH, HJS, MZ25, WWZ22}, and some authors extend this gluing techniques
form cotorsion
pairs to $n$-cotorsion pairs \cite{CW, HJ, MX}.
In this paper, we continue to explore the conditions under which we can
glue (complete hereditary) cotorsion
pairs along recollements of abelian categories,
where the known gluing techniques are valid only under the assumption that both $i^*$ and $i^!$ are exact,
or just $i^!$ is exact; see \cite{FH, HJS, MZ25, WWZ22}.
However, this assumption appears quite restrictive,
because recollements of abelian categories where
$i^!$ is exact only can exist in the case of comma categories or triangular matrix rings
\cite{FV, GKP21}. It is well
known that the exactness of $i^!$ (resp. $i^*$) implies the exactness of $j_*$
(resp. $j_!$); see for example \cite{FZ17, TO79}.
In this paper, we generalize the known gluing techniques on cotorsion pairs and recollements of abelian categories
by relaxing the exactness condition on $i^!$ or $i^*$, requiring only that the derived functor  $\mathbb{R}_1j_*$ or $\mathbb{L}_1j_!$ vanish on certain objects.
 To state our results precisely, we first introduce some notation.

Throughout this paper,
all abelian categories are assumed to have enough projective and injective objects.
Let $(  \mathcal{A^{'}}, \mathcal{A}, \mathcal{A^{''}}, i^\ast, i_\ast, i^!, j_!, j^\ast, j_\ast)$
be a recollement of abelian categories, and let ${\mathcal{X}}$ (resp. ${\mathcal{Y}}$)
be a subcategory of ${\mathcal{A}'}$ (resp. ${\mathcal{A}''}$).
Motivated by Hu-Zhu-Zhu
\cite{HJS}, we set
\begin{center}

$\mathcal{N}^{\mathcal{X}}_{\mathcal{Y}}=\{X\in\mathcal{A}\ |\ i^!X\in{\mathcal{X}},\ j^*X\in{\mathcal{Y}}
,\ (\mathbb{R}_1i^!)(X)=0\},$

$\mathcal{M}^{\mathcal{X}}_{\mathcal{Y}}=\{X\in\mathcal{A}\ |\ i^*X\in\mathcal{X} ,\ j^*X\in\mathcal{Y},\ (\mathbb{L}_1i^*)(X)=0\}.$

\end{center}
For a class of objects $\mathcal{C}$ in ${\mathcal{A}}$, we
denote by $\mathcal{C}^\perp$ the class of $X$ in $\mathcal{A}$ such that $\Ext^{1}_{\mathcal{A}}(C,X)=0$ for all $C\in\mathcal{C}$,
and the notation $^{\bot}\mathcal{C}$ is defined similarly.

\begin{theorem} {\rm (Theorem~\ref{3.3})} \label{1.1}
  Assume that $(\mathcal{U^{'}},\mathcal{V^{'}})$ and $(\mathcal{U^{''}},\mathcal{V{''}})$ are cotorsion pairs in $\mathcal{A^{'}}$ and $\mathcal{A^{''}}$, respectively. If $(\mathbb{L}_1j_!)(\mathcal{U^{''}})=0$ then
 $(^{\bot}{\mathcal{N}_{\mathcal{V^{''}}}^{\mathcal{V^{'}}}},
\mathcal{N}_{\mathcal{V^{''}}}^{\mathcal{V^{'}}})$ is a cotorsion pair in $\mathcal{A}$;
If $(\mathbb{R}_1j_*)(\mathcal{V}^{''})=0$ then
$(\mathcal{M}_{\mathcal{U^{''}}}^{\mathcal{U^{'}}},
({\mathcal{M}_{\mathcal{U^{''}}}^{\mathcal{U^{'}}}})^\bot)$
is a cotorsion pair in $\mathcal{A}$.
\end{theorem}

Generally, it is non-trivial to characterize the classes $^{\bot}{\mathcal{N}_{\mathcal{V^{''}}}^{\mathcal{V^{'}}}}$ and $({\mathcal{M}_{\mathcal{U^{''}}}^{\mathcal{U^{'}}}})^\bot$. Our next goal is to determine when
$^{\bot}{\mathcal{N}_{\mathcal{V^{''}}}^{\mathcal{V^{'}}}}
=\mathcal{M}_{\mathcal{V^{''}}}^{\mathcal{V^{'}}}$
and $({\mathcal{M}_{\mathcal{U^{''}}}^{\mathcal{U^{'}}}})^\bot=\mathcal{N}_{\mathcal{U^{''}}}^{\mathcal{U^{'}}}$, which would imply that $(\mathcal{M}_{\mathcal{U^{''}}}^{\mathcal{U^{'}}}, \mathcal{N}_{\mathcal{V^{''}}}^{\mathcal{V^{'}}})$
is a cotorsion pair. For this purpose, we add a stronger assumption on the recollement.
We say that a recollement of abelian categories $(\mathcal{A^{'}}, \mathcal{A}, \mathcal{A^{''}},
 i^\ast, i_\ast, i^!, j_!, j^\ast, j_\ast)$
{\it satisfies the condition (P)} if
$\varepsilon_P$ is a monomorphism for any projective object $P$ in $\mathcal{A}$,
where $\varepsilon :j_!j^*\rightarrow 1_{\mathcal{A}}$
is the counit of the adjoint pair $(j_!, j^*)$.
Under this condition, we prove in Proposition~\ref{prop-two-contor-equal} that the equality
$(^{\bot}{\mathcal{N}_{\mathcal{V^{''}}}^{\mathcal{V^{'}}}},
\mathcal{N}_{\mathcal{V^{''}}}^{\mathcal{V^{'}}})=(\mathcal{M}_{\mathcal{U^{''}}}^{\mathcal{U^{'}}},
({\mathcal{M}_{\mathcal{U^{''}}}^{\mathcal{U^{'}}}})^\bot)$ holds.
Additionally, we establish the following result concerning the heredity and completeness of $(\mathcal{M}_{\mathcal{U^{''}}}^{\mathcal{U^{'}}}, \mathcal{N}_{\mathcal{V^{''}}}^{\mathcal{V^{'}}})$.

\begin{theorem} {\rm (Theorem~\ref{3.8(1)})} \label{1.2}
Let $(  \mathcal{A^{'}},\mathcal{A},\mathcal{A^{''}},i^\ast,i_\ast,i^!,j_!,j^\ast,j_\ast)$ be a recollement of abelian categories satisfying condition (P). Assume that $(\mathcal{U^{'}},\mathcal{V^{'}})$ and $(\mathcal{U^{''}},\mathcal{V{''}})$ are complete hereditary cotorsion pairs in $\mathcal{A^{'}}$ and $\mathcal{A^{''}}$, respectively. If $(\mathbb{L}_1j_!)(\mathcal{U^{''}})=0$ and $(\mathbb{R}_1j_*)(\mathcal{V}^{''})=0$,
then $(\mathcal{M}_{\mathcal{U^{''}}}^{\mathcal{U^{'}}}, \mathcal{N}_{\mathcal{V^{''}}}^{\mathcal{V^{'}}})$ is a complete hereditary cotorsion pair in $\mathcal{A}$.
\end{theorem}

Recollements of abelian categories satisfying condition (P) arise naturally in representation theory.
First, such a condition is satisfied whenever
$i^*$ or $i^!$ is exact; see Proposition~\ref{pro-exact-recoll}. Consequently, Theorem~\ref{1.2} extends
the known gluing methods on cotorsion pairs and recollements of abelian categories \cite{CW, HJS} by weakening the exactness requirement on
$i^!$ or $i^*$; see Remark~\ref{rek-exact-recoll}. Second, a triangular matrix ring yields
two recollements of module categories which satisfy the condition (P), and then
Theorem~\ref{1.1}
and Theorem~\ref{1.2} recover and strength some results on cotorsion pairs and triangular matrix rings in \cite{ML, ZR};
see Remark~\ref{rek-zhu-Mao-cotor}.
Finally, the recollement associated with Morita ring $\Lambda_{(\phi,\psi)}={\scriptsize\left( \begin{array} {cc}
	A	&_AN_B  \\
_BM_A	& B
	\end{array}\right)  }$ satisfies condition (P) whenever $\phi$ or $\psi$ is a monomorphism,
thereby allowing Theorem~\ref{1.2} to be applied for constructing new cotorsion pairs on such rings; see Corollary~\ref{cor-motia-contor-1}
and~\ref{cor-motia-contor-2}. In particular, when $M\otimes _AN=0$ or $N\otimes _BM=0$,
we construct some explicit cotorsion pairs distinct from those in \cite{ZP}; see Corollary~\ref{cor-motia-contor-3}
and~\ref{cor-motia-contor-4}.

	This paper is organized as follows. In section \ref{Section-definitions and conventions}, we give some
	definitions and some preliminary results which are needed for our proof.
	In section \ref{Gluing cotorsion pairs along recollements}, we glue cotorsion pairs
by recollements of abelian categories, and we prove Theorem~\ref{1.1} and Theorem~\ref{1.2} mentioned in the introduction.
In section \ref{4}, we give some applications in Morita rings and triangular matrix rings.

	\section{\large Preliminaries}\label{Section-definitions and conventions}
	
	\indent\indent Throughout $\mathcal{A}$ is an abelian category with enough projective and injective objects.
We denote by $\mathcal{P_{\mathcal{A}}}$ the class of projective objects of $\mathcal{A}$ and by $\mathcal{I_{\mathcal{A}}}$ the class of injective objects of $\mathcal{A}$.
The composition of two homomorphisms $f:X\rightarrow Y$ and $g:Y\rightarrow Z$ will be denoted by $gf$.
Throughout this paper, all subcategories are assumed to be full subcategories which are closed under isomorphisms.

First, we recall the definition of recollement of abelian categories; see for instance \cite{BBD82, FV}.
	\begin{definition}\label{reco} {\rm Let $\mathcal{A}$, $\mathcal{A^{'}}$ and $\mathcal{A^{''}}$ be abelian categories. A {\it recollement} of $\mathcal{A}$ relative
			to $\mathcal{A^{'}}$ and $\mathcal{A^{''}}$
is given by the diagram
		$$\xymatrix@!=4pc{ \mathcal{A}^{'} \ar[r]|{i_{\ast}} & \mathcal{A} \ar@<-3ex>[l]_{i^{\ast}}
			\ar@<+3ex>[l]^{i^!} \ar[r]|{j^{\ast}} & \mathcal{A}^{''}
			\ar@<-3ex>[l]_{j_!} \ar@<+3ex>[l]^{j_{\ast}}}  $$
 henceforth denoted by $( \mathcal{A^{'}},\mathcal{A},\mathcal{A^{''}},i^\ast,i_\ast,i^!,j_!,j^\ast,j_\ast)$, satisfying the following conditions:
		
		(1) $(i^{\ast},i_{\ast}), (i_{\ast},i^!), (j_{!},j^{\ast})$ and $(j^{\ast},j_{\ast})$ are adjoint
		pairs;

		(2) $i_{\ast}$, $j_!$ and $j_{\ast}$ are fully faithful functors;

		(3) $\Ker j^{\ast}=\Im i_{\ast}$ .
		
		 }
	\end{definition}

In the following remark we isolate some easily
 established properties of a recollement situation which will be useful later.
		\begin{remark}\label{2.4}
		{\rm (1) The functors $i_{\ast}$ and $j^{\ast}$ are exact,
$i^*$ and $j_!$ are right exact, and $i^!$ and $j_*$ are left exact.
This indicates that $i^*$ and $j_!$ preserve projective objects, and $i^!$ and $j_*$ preserve injective objects.

		(2) There are four natural isomorphisms $i^{\ast}i_{\ast}\cong{1 _{\mathcal{A}'}},
i^{!}i_{\ast}\cong{1  _{\mathcal{A}'}},j^{\ast}j_{!}\cong{1  _{\mathcal{A}''}}$ and $j^{\ast}j_{\ast}\cong{1  _{\mathcal{A}''}}$
given by the units and the counits of the adjoint
 pairs, respectively.
		
		(3) (See \cite[Proposition 2.6]{Psa14} or \cite[Lemma 3.1]{FZ17})
For any $A\in{\mathcal{A}}$, there exist $M^{'}, N^{'}\in \mathcal{A^{'}}$ such that the
		units and counits induce the following two exact sequences
		$$0\rightarrow i_{\ast}(M^{'}) \rightarrow j_{!}j^{\ast}(A) \overset{\varepsilon_A}{\rightarrow} A \rightarrow i_{\ast}i^{\ast}(A) \rightarrow0,$$
		$$0\rightarrow i_{\ast}i^{!}(A) \rightarrow A \overset{\delta_A}{\rightarrow} j_{\ast}j^{\ast}(A)  \rightarrow i_{\ast}(N^{'}) \rightarrow0.$$
		
(4) It follows from the adjunction that $\Im j_!\subseteq \Ker i^*$ and $\Im j_*\subseteq \Ker i^!$. }

	\end{remark}

Now we recall the notion of cotorsion pairs.
\begin{definition}\label{cotorsion pair}{\rm(\cite{SL})} {\rm
A pair ($\mathcal{C},\mathcal{D}$) of classes of objects in an abelian
category $\mathcal{A}$ is called a {\it cotorsion pair} if $\mathcal{C}^\perp=\mathcal{D}$ and $^{\perp}\mathcal{D}=\mathcal{C}$. Here $\mathcal{C}^\perp$ is the class of $X$ in $\mathcal{A}$ such that $\Ext^{1}_{\mathcal{A}}(C,X)=0$ for all $C\in\mathcal{C}$, and similarly $^{\perp}\mathcal{D}$ is the class of $Y$ in $\mathcal{A}$ such that $\Ext^{1}_{\mathcal{A}}(Y,D)=0$ for all $D\in\mathcal{D}$.}
\end{definition}

From \cite{EEE, GR}, a cotorsion pair ($\mathcal{C},\mathcal{D}$) is {\it complete} if for each $M\in\mathcal{A}$, there exist
two short exact sequences
\begin{center}
  $\xymatrix@C=1.5pc{0\ar[r]& D_1 \ar[r]& C_1\ar[r]^{f} & M\ar[r]  & 0}$,  $\xymatrix@C=1.5pc{0\ar[r]& M \ar[r]^{g}& D_2\ar[r] &C_2\ar[r]  & 0}$
\end{center}
with $C_1,C_2\in \mathcal{C}$, $D_1,D_2\in \mathcal{D}$.

\begin{proposition}\label{co1}{\rm(\cite[Proposition 7.17]{EEE})}
{\it Let $(\mathcal{C},\mathcal{D})$ be a cotorsion pair in $\mathcal{A}$.
 Then the following are equivalent:

 {\rm(1)} $(\mathcal{C},\mathcal{D})$ is complete;

 {\rm(2)} For any $X\in\mathcal{A}$, there is a short exact sequence $0\rightarrow D\rightarrow C \rightarrow X \rightarrow0$ with $C\in\mathcal{C}$ and $D\in\mathcal{D}$;

 {\rm(3)} For any $X\in\mathcal{A}$, there is a short exact sequence $0\rightarrow X\rightarrow D^{'} \rightarrow C^{'} \rightarrow0$ with $C^{'}\in\mathcal{C}$ and $D^{'}\in\mathcal{D}$.
}
\end{proposition}

From \cite{GR}, a cotorsion pair ($\mathcal{C},\mathcal{D}$) is {\it hereditary} if
$\mathcal{C}$ is closed under the kernel of epimorphisms,
and $\mathcal{D}$ is closed under the cokernel of monomorphisms.
For hereditary cotorsion pairs, we have the following proposition.

\begin{proposition}\label{he1}{\rm(\cite[Theorem 1.2.10]{JR})}
{ \it Let $(\mathcal{C},\mathcal{D})$ be a cotorsion pair in $\mathcal{A}$.
 Then the following are equivalent:

 {\rm(1)} $(\mathcal{C},\mathcal{D})$ is hereditary;

 {\rm(2)} $\mathcal{C}$ is closed under the kernel of epimorphisms;

 {\rm(3)} $\mathcal{D}$ is closed under the cokernel of monomorphisms;

 {\rm(4)} $\Ext^2_{\mathcal{A}}(\mathcal{C},\mathcal{D})=0$;

  {\rm(5)} $\Ext^i_{\mathcal{A}}(\mathcal{C},\mathcal{D})=0$ for $i\geq1$.
  }
\end{proposition}




	\section{Gluing cotorsion pairs along recollements}\label{Gluing cotorsion pairs along recollements}
	
\indent\indent Let $(  \mathcal{A^{'}},\mathcal{A},\mathcal{A^{''}},i^\ast,i_\ast,i^!,j_!,j^\ast,j_\ast)$
be a recollement of abelian categories. In this section, we will investigate how to glue together cotorsion pairs in
 $\mathcal{A^{'}}$ and $\mathcal{A^{''}}$ to obtain a cotorsion pair in $\mathcal{A}$.
Let's begin with the following lemma.
	
\begin{lemma}\label{3.1} For any $X\in \mathcal{A}$, $L\in \mathcal{A^{'}}$ and $Y\in \mathcal{A^{''}}$, we have

{\rm(1)} If $(\mathbb{R}_1i^!)(X)=0$, then $\Ext^1_{\mathcal{A}}(i_*L,X)\cong \Ext^1_{\mathcal{A^{'}}}(L,i^!X)$.

{\rm(2)} If $(\mathbb{L}_1i^*)(X)=0$, then $\Ext^1_{\mathcal{A^{'}}}(i^*X,L)\cong \Ext^1_{\mathcal{A}}(X,i_*L)$.

{\rm(3)} If $(\mathbb{R}_1j_*)(Y)=0$, then $\Ext^1_{\mathcal{A^{''}}}(j^*X,Y)\cong \Ext^1_{\mathcal{A}}(X,j_*Y)$.

{\rm(4)} If $(\mathbb{L}_1j_!)(Y)=0$, then $\Ext^1_{\mathcal{A}}(j_!Y,X)\cong \Ext^1_{\mathcal{A^{''}}}(Y,j^*X)$.
	\end{lemma}
	\begin{proof}
We only prove (1). The assertions (2), (3) and (4) can be proved similarly.

(1) Since $\mathcal{A}$ has enough injective objects, there is a short exact sequence $\xymatrix@C=1.5pc{ 0\ar[r] & X\ar[r] & I\ar[r]& C\ar[r]&0}$ in $\mathcal{A}$ with $I\in \mathcal{I_{\mathcal{A}}}$. Applying the left exact functor $i^!$ and using $(\mathbb{R}_1i^!)(X)=0$, we get
the following short exact sequence
\begin{center}

$\xymatrix@C=2pc{ 0\ar[r] & i^!X\ar[r] & i^!I\ar[r]& i^!C
\ar[r]&0}.$
\end{center}
Since $i^!$ preserves injective objects, it follows from \cite[Lemma 3.7 (2)]{ZP}
(which also holds for abelian categories) that $\Ext^1_{\mathcal{A}}(i_*L,X)\cong \Ext^1_{\mathcal{A^{'}}}(L,i^!X)$.
\end{proof}
	
In \cite{ZR}, Zhu-Peng-Ding construct four cotorsion pairs over triangular matrix rings.
Motivated by \cite{ZR}, and noting that a triangular matrix ring gives rise to a
special recollement of abelian categories, we construct cotorsion pairs along an arbitrary recollement of abelian categories.
Let ${\mathcal{X}}$ and ${\mathcal{Y}}$ be two classes of objects in $\mathcal{A'}$ and $\mathcal{A''}$,
respectively.
Recall that in our setting \begin{center}

$\mathcal{N}^{\mathcal{X}}_{\mathcal{Y}}=\{X\in\mathcal{A}\ |\ i^!X\in{\mathcal{X}},\ j^*X\in{\mathcal{Y}}
,\ (\mathbb{R}_1i^!)(X)=0\},$

$\mathcal{M}^{\mathcal{X}}_{\mathcal{Y}}=\{X\in\mathcal{A}\ |\ i^*X\in\mathcal{X} ,\ j^*X\in\mathcal{Y},\ (\mathbb{L}_1i^*)(X)=0\}.$

\end{center}
Inspired by \cite[Proposition 3.3]{ZR},
we set $\mathcal{C}_{\mathcal{Y}}^{\mathcal{X}}=i_*({\mathcal{X}})\cup{j_!(\mathcal{Y})}$
and
$\mathcal{D}_{\mathcal{Y}}^{\mathcal{X}}=i_*({\mathcal{X}} )\cup{j_*(\mathcal{Y})}$,
and we have the following lemma.

	\begin{lemma}\label{lem-syz}
{\rm(1)} Assume that $\mathcal{U'}$ and $\mathcal{U''}$ are classes of objects in $\mathcal{A'}$ and $\mathcal{A''}$,
respectively. If $(\mathbb{L}_1j_!)({\mathcal{U^{''}}})=0$ and $\mathcal{P_{\mathcal{A^{'}}}}\subseteq {\mathcal{U^{'}}}$,
then $(\mathcal{C}_{\mathcal{U''}}^{\mathcal{U'}})^{\perp}=\mathcal{N}_{\mathcal{(U^{''})^\perp}}^{\mathcal{(U^{'})^\perp}}$.	

{\rm(2)} Assume that $\mathcal{V'}$ and $\mathcal{V''}$ are classes of objects in $\mathcal{A'}$ and $\mathcal{A''}$, respectively.
If $(\mathbb{R}_1j_*)(\mathcal{V}^{''})=0$ and $\mathcal{I_{\mathcal{A^{'}}}}\subseteq {\mathcal{V^{'}}}$, then $^{\perp}(\mathcal{D}_{\mathcal{V''}}^{\mathcal{V'}})=\mathcal{M}_{\mathcal{^\perp(V^{''})}}^{\mathcal{^\perp(V^{'})}}$	.	
	\end{lemma}	

	\begin{proof}
We only prove (1), and the proof of (2) is similar.
First, we show that $(\mathcal{C}_{\mathcal{U''}}^{\mathcal{U'}})^{\perp}\subseteq \mathcal{N}_{\mathcal{(U^{''})^\perp}}^{\mathcal{(U^{'})^\perp}}$.
For any $X\in (\mathcal{C}_{\mathcal{U''}}^{\mathcal{U'}})^{\perp}=(i_*(\mathcal{U^{'}})\cup j_!({\mathcal{U{''}}}))^{\bot}$,
we have $X\in(i_*(\mathcal{U^{'}}))^{\perp} $ and $X\in(j_!({\mathcal{U{''}}}))^{\perp} $.
To show $X\in \mathcal{N}_{\mathcal{(U^{''})^\perp}}^{\mathcal{(U^{'})^\perp}}$, it suffices to prove that
$(\mathbb{R}_1i^!)(X)=0$, $i^!X\in ({\mathcal{U^{'}}})^{\bot}$ and $j^*X \in ({\mathcal{U{''}}})^{\bot}$.
Consider the short exact sequence $\xymatrix@C=1.5pc{ 0\ar[r] & X\ar[r] & I\ar[r]^{g}& C\ar[r]&0}$ with $I\in \mathcal{I_{\mathcal{A}}}.$
Since $\mathcal{A^{'}}$ have enough projective objects, there is an epimorphism $ f:P_1\rightarrow i^!C$
with $P_1\in \mathcal{P_{\mathcal{A^{'}}}}$. Since $\mathcal{P_{\mathcal{A^{'}}}}\subseteq {\mathcal{U^{'}}}$
and $X\in(i_*(\mathcal{U^{'}}))^{\perp} $, we have that
$\Ext^1_{\mathcal{A}^{'}}(i_*(P_1),X)=0$, and then we get the following short exact sequence
$$\xymatrix@C=2pc{ 0\ar[r] & \Hom _{\mathcal{A}}(i_*(P_1),X)\ar[r] &  \Hom_{\mathcal{A}}(i_*(P_1),I)\ar[r]^{g_*}&  \Hom_{\mathcal{A}}(i_*(P_1),C)\ar[r]&0}.$$
It follows from the adjunction that the sequence
$$\xymatrix@C=2pc{ 0\ar[r] & \Hom_{\mathcal{A}'}(P_1,i^!X)\ar[r] &  \Hom_{\mathcal{A}'}(P_1,i^!I)\ar[r]^{(i^!g)_*}&  \Hom_{\mathcal{A}'}(P_1,i^!C)\ar[r]&0}$$
is also exact. Therefore,
$(i^!g)_*$ is an epimorphism,
and then there exists $ u\in \Hom(P_1,i^!I) $ such that $ f=(i^!g)_*(u)=(i^!g)u$.
Since $f$ is an epimorphism, we infer that $i^!g$ is also an epimorphism,
and then $(\mathbb{R}_1i^!)(X)=0$.
For any $A\in \mathcal{U'}$, it follows from $X\in(i_*(\mathcal{U^{'}}))^{\perp} $ that $\Ext^1_{\mathcal{A}}(i_*A,X)=0$,
and then $\Ext^1_{\mathcal{A}^{'}}(A,i^!X)=0$
by Lemma~\ref{3.1} (1). Therefore, we get $i^!X\in ({\mathcal{U^{'}}})^{\bot}$.
For any $ C^{'}\in {\mathcal{U{''}}}$, it follows from $(\mathbb{L}_1j_!)({\mathcal{U^{''}}})=0$ and Lemma~\ref{3.1} (4)
that $\Ext^1_{\mathcal{A}^{''}}(C^{'},j^*X)\cong \Ext^1_{\mathcal{A}}(j_!C^{'},X)$,
which is zero since $X\in(j_!({\mathcal{U{''}}}))^{\perp} $.
Therefore, we get $j^*X \in ({\mathcal{U{''}}})^{\bot}$.

Now we show that $\mathcal{N}_{\mathcal{(U^{''})^\perp}}^{\mathcal{(U^{'})^\perp}}
\subseteq(\mathcal{C}_{\mathcal{U''}}^{\mathcal{U'}})^{\perp}$.
For any $ X\in \mathcal{N}_{\mathcal{(U^{''})^\perp}}^{\mathcal{(U^{'})^\perp}}$, we have
 $i^!X\in (\mathcal{U^{'}})^{\bot}$, $j^*X\in (\mathcal{U}^{''})^\perp $ and $(\mathbb{R}_1i^!)(X)=0$.
 For any $A\in \mathcal{U^{'}}$, it follows from Lemma~\ref{3.1} (1) that $\Ext^1_{\mathcal{A}}(i_*A,X)\cong \Ext^1_{\mathcal{A}^{'}}(A,i^!X)$,
which is zero since $i^!X\in (\mathcal{U^{'}})^{\bot}$.
Therefore, we get $X\in(i_*({\mathcal{U^{'}}}))^{\perp} $.
  For any $ C^{'}\in {\mathcal{U{''}}}$, since $(\mathbb{L}_1j_!)(\mathcal{U}^{''})=0$,
it follows from Lemma~\ref{3.1} (4) that
 $\Ext^1_{\mathcal{A}}(j_!C^{'},X) \cong\Ext^1_{\mathcal{A}^{''}}(C^{'},j^*X)$,
 which is zero since $j^*X\in (\mathcal{U}^{''})^\perp $.
Therefore, we get $X\in (j_!({\mathcal{U{''}}}))^{\bot}$, and thus
$X\in (i_*(\mathcal{U^{'}})\cup j_!({\mathcal{U{''}}}))^{\bot}=(\mathcal{C}_{\mathcal{U''}}^{\mathcal{U'}})^{\perp}$
	\end{proof}

Motivated by \cite[Theorem 3.4]{ZR}, we construct two cotorsion pairs in $\mathcal{A}$ in the next theorem.

\begin{theorem}\label{3.3}
Assume that $(\mathcal{U^{'}},\mathcal{V^{'}})$ and $(\mathcal{U^{''}},\mathcal{V{''}})$ are cotorsion pairs in $\mathcal{A^{'}}$ and $\mathcal{A^{''}}$, respectively. If $(\mathbb{L}_1j_!)(\mathcal{U^{''}})=0$ then
 $(^{\bot}{\mathcal{N}_{\mathcal{V^{''}}}^{\mathcal{V^{'}}}},
\mathcal{N}_{\mathcal{V^{''}}}^{\mathcal{V^{'}}})$ is a cotorsion pair in $\mathcal{A}$;
If $(\mathbb{R}_1j_*)(\mathcal{V}^{''})=0$ then
$(\mathcal{M}_{\mathcal{U^{''}}}^{\mathcal{U^{'}}},
({\mathcal{M}_{\mathcal{U^{''}}}^{\mathcal{U^{'}}}})^\bot)$
is a cotorsion pair in $\mathcal{A}$.
	\end{theorem}

\begin{proof}
Let $\mathcal{C}$ be a class of objects in $\mathcal{A}$.
It is easy to check that both $(^\perp(\mathcal{C}^\perp),\mathcal{C}^\perp)$
and $(^\perp \mathcal{C}, (^\perp \mathcal{C} )^\perp   )$ are cotorsion pairs in $\mathcal{A}$.
Therefore, we have that both
$(^{\bot}{((\mathcal{C}_{\mathcal{U''}}^{\mathcal{U'}})^{\perp})},
(\mathcal{C}_{\mathcal{U''}}^{\mathcal{U'}})^{\perp})$ and
$(^{\perp}(\mathcal{D}_{\mathcal{V''}}^{\mathcal{V'}}),
(^{\perp}(\mathcal{D}_{\mathcal{V''}}^{\mathcal{V'}}))^\bot)$ are cotorsion pairs.
Since $(\mathcal{U^{'}},\mathcal{V^{'}})$ and $(\mathcal{U^{''}},\mathcal{V{''}})$ are cotorsion pairs,
we get $\mathcal{P_{\mathcal{A^{'}}}}\subseteq {\mathcal{U^{'}}}$ and $\mathcal{I_{\mathcal{A^{'}}}}\subseteq {\mathcal{V^{'}}}$
by definition.
If $(\mathbb{L}_1j_!)(\mathcal{U^{''}})=0$, then we get
$(\mathcal{C}_{\mathcal{U''}}^{\mathcal{U'}})^{\perp}=
\mathcal{N}_{\mathcal{V^{''}}}^{\mathcal{V^{'}}}$ by Lemma~\ref{lem-syz} (1), and thus
$(^{\bot}{\mathcal{N}_{\mathcal{V^{''}}}^{\mathcal{V^{'}}}},
\mathcal{N}_{\mathcal{V^{''}}}^{\mathcal{V^{'}}})$ is a cotorsion pair in $\mathcal{A}$.
If $(\mathbb{R}_1j_*)(\mathcal{V}^{''})=0$, then we get
$^{\perp}(\mathcal{D}_{\mathcal{V''}}^{\mathcal{V'}})=\mathcal{M}_{\mathcal{U^{''}}}^{\mathcal{U^{'}}}$
by Lemma~\ref{lem-syz} (2),
and thus
$(\mathcal{M}_{\mathcal{U^{''}}}^{\mathcal{U^{'}}},
({\mathcal{M}_{\mathcal{U^{''}}}^{\mathcal{U^{'}}}})^\bot)$
is a cotorsion pair in $\mathcal{A}$.
\end{proof}

Next, we will discuss when
$(^{\bot}{\mathcal{N}_{\mathcal{V^{''}}}^{\mathcal{V^{'}}}},
\mathcal{N}_{\mathcal{V^{''}}}^{\mathcal{V^{'}}})=(\mathcal{M}_{\mathcal{U^{''}}}^{\mathcal{U^{'}}},
({\mathcal{M}_{\mathcal{U^{''}}}^{\mathcal{U^{'}}}})^\bot)$.
We need the following lemmas.
	
\begin{lemma}\label{3.4} Let $X$ be an object in $\mathcal{A}$.

			{\rm(1)} If the unit $X\overset{\delta_X}{\rightarrow} j_*j^*X$ is an epimorphism then $(\mathbb{R}_1i^!)(X)=0$,
and the converse also holds if $\delta_I$ is an epimorphism for any $ I\in \mathcal{I_{\mathcal{A}}}$.
			
	{\rm(2)} If the counit $j_!j^*X\overset{\varepsilon_X}{\rightarrow} X$ is a monomorphism then $(\mathbb{L}_1i^*)(X)=0$,
and the converse also holds if $\varepsilon_P$ is a monomorphism for any $ P\in \mathcal{P_{\mathcal{A}}}$.
	\end{lemma}	
	\begin{proof}
We only prove (1), and the proof of (2) is similar.
Consider the short exact sequence $\xymatrix@C=1.5pc{ 0\ar[r] & X\ar[r]^{f} & I\ar[r]^{g}& C\ar[r]&0}$ in $\mathcal{A}$ with $I\in \mathcal{I_{\mathcal{A}}}$. By Remark~\ref{2.4} (3), we get the following exact commutative diagram:
$$\xymatrix@C=2pc{ &0\ar[d]&0\ar[d]&0 \ar[d]&\\ 0\ar[r] & i_*i^!X   \ar[r]^{i_*i^!f} \ar[d] & i_*i^!I
			\ar[r]^{i_*i^!g} \ar[d] &  i_*i^!C\ar[d] &  \\
			0\ar[r]& X \ar[r]^{f}\ar[d]^{\delta_X}& I\ar[r]^{g}\ar[d]^{\delta_I} & C \ar[r]\ar[d]^{\delta_C}  & 0\\0\ar[r]&j_*j^*X\ar[r]&j_*j^*I\ar[r]&j_*j^*C.}$$
If $\delta_X$ is an epimorphism, then it follows from the Snake lemma that the sequence
$$\xymatrix@C=2pc{ 0\ar[r] & i_*i^!X\ar[r]^{i_*i^!f} & i_*i^!I
			\ar[r]^{i_*i^!g} &  i_*i^!C \ar[r]& 0}$$ is exact.
Applying $i^*$, we obtain the following exact sequence $$\xymatrix@C=2pc{i^* i_*i^!X\ar[r]^{i^*i_*i^!f} & i^*i_*i^!I
			\ar[r]^{i^*i_*i^!g} & i^* i_*i^!C \ar[r]& 0}.$$ Since $i^* i_* \cong 1 _{\mathcal{A'}}$, we get an
exact sequence $\xymatrix@C=2pc{ i^!X\ar[r]^{i^!f} & i^!I
			\ar[r]^{i^!g} &i^!C \ar[r]& 0}$,
and then $(\mathbb{R}_1i^!)(X)=0$.

Conversely, if $(\mathbb{R}_1i^!)(X)=0$ then the sequence $$\xymatrix@C=2pc{0\ar[r]& i^!X\ar[r]^{i^!f} & i^!I
			\ar[r]^{i^!g} &i^!C \ar[r]& 0}$$ is exact. Applying $i_*$, we get
a short exact sequence $$\xymatrix@C=2pc{0\ar[r]& i_*i^!X\ar[r] & i_*i^!I
			\ar[r]^{i_*i^!g} &i_*i^!C \ar[r]& 0}.$$
Therefore, $g$ induces an epimorphism between $\Ker (\delta_I)$ and $\Ker (\delta_C)$.
On the other hand, it follows from the Snake lemma that the sequence
$$\Ker (\delta_I)\rightarrow \Ker (\delta_C)\rightarrow \Coker \delta_X \rightarrow \Coker \delta_I $$
is exact.
Since $\delta_I$ is an epimorphism, we get $\Coker\delta_I=0$, which implies that $\Coker \delta_X=0$.
Therefore, $\delta_X$ is an epimorphism.
\end{proof}






The following observation is crucial for what follows.
\begin{lemma}\label{P}
$j_!j^*P\overset{\varepsilon_P}{\rightarrow} P$ is a monomorphism for any $ P\in \mathcal{P_{\mathcal{A}}}$ if and only if $I\overset{\delta_I}{\rightarrow}j_*j^*I$ is an epimorphism for any $ I\in \mathcal{I_{\mathcal{A}}}$.
\end{lemma}
\begin{proof}
Assume that $j_!j^*P\overset{\varepsilon_P}{\rightarrow} P$ is a monomorphism for any $ P\in \mathcal{P_{\mathcal{A}}}$.
For any $ I\in \mathcal{I_{\mathcal{A}}}$, let $\pi:P\twoheadrightarrow j_*j^*I$ be the projective cover of $j_*j^*I$.
By the adjunction, we obtain the following commutative diagram:
 \begin{center}
 $\xymatrix{\Hom _{\mathcal{A}}(j_!j^*P,I)\ar[rr]^{(\delta_I)_*}\ar[d]^{\cong}&&\Hom_{\mathcal{A}''}(j_!j^*P,j_*j^*I)\ar[d]^{\cong}\\
 \Hom_{\mathcal{A}}(j^*P,j^*I)\ar[rr]^{(j^*\delta_I)_*}&&\Hom_{\mathcal{A}''}(j^*P,j^*j_*j^*I).}$
 \end{center}
  Let $\xi:j^*j_*\rightarrow 1_{\mathcal{A''}}$ be counit of the adjoint pair $(j^*,j_*)$.
 Since $\delta : 1_{\mathcal{A}} \rightarrow j_*j^*$ is the unit,
 we have $(\xi_{j^*I})\circ ( j^*\delta_I)=1 _{j^*I}$,
and then $(\xi_{j^*I}\circ j^*\delta_I)_*=1_{\Hom_{\mathcal{A}}(j^*P,j^*I)}$.
Therefore, we get $(\xi_{j^*I})_*\circ(j^*\delta_I)_*=1_{\Hom_{\mathcal{A}}(j^*P,j^*I)}$.
On the other hand, it follows from Remark~\ref{2.4} (2) that $\xi$ is an natural isomorphism,
and then $(\xi_{j^*I})_*$ is an isomorphism. Therefore, $(j^*\delta_I)_*$ is an isomorphism, and
so is $(\delta_I)_*$ by the above commutative diagram.
Consider the homomorphism $\pi\varepsilon_P\in\Hom_{\mathcal{A}}(j_!j^*P,j_*j^*I)$.
Since $(\delta_I)_*$ is an isomorphism,
there exists $f\in\Hom_{\mathcal{A}}(j_!j^*P,I)$ such that $(\delta_I)_*(f)=\pi\varepsilon_P$.
Therefore, we get $\delta_I f=\pi\varepsilon_P$.
Consider the following commutative diagram:
 \begin{center}
 $\xymatrix{j_!j^*P\ar@{>->}[r]^{\varepsilon_P}\ar[d]^{f}&P\ar[d]^{\pi}\ar@{-->}[ld]|{g}\\I\ar[r]|{\delta_I}&j_*j^*I.}$
 \end{center}
 Since $ I\in \mathcal{I_{\mathcal{A}}}$, there exists $g\in\Hom_{\mathcal{A}}(P,I)$ such that $f=g\varepsilon_P$. Then
  it follows that $\pi\varepsilon_P=\delta_I f=\delta_I g\varepsilon_P$, that is, $(\pi-\delta_I g)\varepsilon_P=0$.
 On the other hand, we have $\Hom_{\mathcal{A''}}(j^*P,j^*I)\cong\Hom_{\mathcal{A}}(P,j_*j^*I)\overset{(\varepsilon_P)^*}{\longrightarrow} \Hom_{\mathcal{A}}(j_!j^*P,j_*j^*I)\cong\Hom_{\mathcal{A''}}(j^*P,j^*I)$ by the adjunction.
Just as we proved that $(\delta_I)_*$ is an isomorphism, we can similarly show that
 $(\varepsilon_P)^*$
  is also an isomorphism,
 and then $\pi-\delta_I g=0$
 since $(\varepsilon_P)^*(\pi-\delta_I g)=(\pi-\delta_I g)\varepsilon_P=0$. Therefore,
 we obtain $\pi=\delta_I g$. Note that $\pi$ is an epimorphism, and then so is $\delta_I$.
The converse can be proved dually.
\end{proof}	

For convenience, we say that a recollement of abelian categories $(\mathcal{A^{'}}, \mathcal{A}, \mathcal{A^{''}},$
\linebreak
$ i^\ast, i_\ast, i^!, j_!, j^\ast, j_\ast)$
{\it satisfies the condition (P)} if
$\varepsilon_P$ is a monomorphism for any $P\in \mathcal{P}_{\mathcal{A}}$,
where $\varepsilon :j_!j^*\rightarrow 1_{\mathcal{A}}$
is the counit of the adjoint pair $(j_!, j^*)$.
In order to give a sufficient condition for $(^{\bot}{\mathcal{N}_{\mathcal{V^{''}}}^{\mathcal{V^{'}}}},
\mathcal{N}_{\mathcal{V^{''}}}^{\mathcal{V^{'}}})=(\mathcal{M}_{\mathcal{U^{''}}}^{\mathcal{U^{'}}},
({\mathcal{M}_{\mathcal{U^{''}}}^{\mathcal{U^{'}}}})^\bot)$, we need the following lemma.

 \begin{lemma}\label{chang}
 Let $\mathcal{U'}$, $\mathcal{V'}$, $\mathcal{U''}$ and $\mathcal{V''}$ be
 classes of objects in $\mathcal{A'}$ and $\mathcal{A''}$,
 respectively. Assume that all of $\mathcal{U'}$, $\mathcal{V'}$, $\mathcal{U''}$ and $\mathcal{V''}$
 contain the zero object. Then $i_*\mathcal{U'}\subseteq \mathcal{M}_{\mathcal{U^{''}}}^{\mathcal{U^{'}}}$, $i_*\mathcal{V'}\subseteq \mathcal{N}_{\mathcal{V^{''}}}^{\mathcal{V^{'}}}$, $j_!\mathcal{U''}\subseteq \mathcal{M}_{\mathcal{U^{''}}}^{\mathcal{U^{'}}}$ and $j_*\mathcal{V''}\subseteq \mathcal{N}_{\mathcal{V^{''}}}^{\mathcal{V^{'}}}$.
 \end{lemma}
 \begin{proof}
  For any $C\in\mathcal{U'}$, to show $i_*C\in \mathcal{M}_{\mathcal{U^{''}}}^{\mathcal{U^{'}}}$,
 it suffices to prove that $i^*(i_*C)\in \mathcal{U'}$, $j^*(i_*C)\in \mathcal{U''}$ and $(\mathbb{L}_1i^*)(i_*C)=0$.
By the property of recollement, we have $i^*i_*C\cong C\in\mathcal{U'}$ and $j^*i_*C=0\in\mathcal{U''}$.
Moreover, it follows \cite[Proposition 4.7]{FV} that
$(\mathbb{L}_1i^*)i_*=0$, and thus $i_*\mathcal{U'}\subseteq \mathcal{M}_{\mathcal{U^{''}}}^{\mathcal{U^{'}}}$. The other
inclusions can be proved similarly.
 \end{proof}

Now we are ready to give a sufficient condition under which $(^{\bot}{\mathcal{N}_{\mathcal{V^{''}}}^{\mathcal{V^{'}}}},
\mathcal{N}_{\mathcal{V^{''}}}^{\mathcal{V^{'}}})=(\mathcal{M}_{\mathcal{U^{''}}}^{\mathcal{U^{'}}},
({\mathcal{M}_{\mathcal{U^{''}}}^{\mathcal{U^{'}}}})^\bot)$. The following proposition is inspired by
\cite[Proposition 3.7]{ZR}.

\begin{proposition}\label{prop-two-contor-equal}
Let $(\mathcal{A^{'}}, \mathcal{A}, \mathcal{A^{''}}, i^\ast, i_\ast, i^!, j_!, j^\ast, j_\ast)$
be a recollement satisfying condition (P).
Assume that $(\mathcal{U^{'}},\mathcal{V^{'}})$ and $(\mathcal{U^{''}},\mathcal{V{''}})$ are
cotorsion pairs in $\mathcal{A^{'}}$ and $\mathcal{A^{''}}$, respectively.
If $(\mathbb{L}_1j_!)(\mathcal{U^{''}})=0$ and
$(\mathbb{R}_1j_*)(\mathcal{V}^{''})=0$,
then $(^{\bot}{\mathcal{N}_{\mathcal{V^{''}}}^{\mathcal{V^{'}}}},
\mathcal{N}_{\mathcal{V^{''}}}^{\mathcal{V^{'}}})=(\mathcal{M}_{\mathcal{U^{''}}}^{\mathcal{U^{'}}},
({\mathcal{M}_{\mathcal{U^{''}}}^{\mathcal{U^{'}}}})^\bot)$.
\end{proposition}

\begin{proof}
It follows from Theorem~\ref{3.3} that both $(^{\bot}{\mathcal{N}_{\mathcal{V^{''}}}^{\mathcal{V^{'}}}},
\mathcal{N}_{\mathcal{V^{''}}}^{\mathcal{V^{'}}})$ and
$(\mathcal{M}_{\mathcal{U^{''}}}^{\mathcal{U^{'}}},
({\mathcal{M}_{\mathcal{U^{''}}}^{\mathcal{U^{'}}}})^\bot)$
 are cotorsion pairs in $\mathcal{A}$. Now we prove $\mathcal{N}_{\mathcal{V^{''}}}^{\mathcal{V^{'}}}=
 ({\mathcal{M}_{\mathcal{U^{''}}}^{\mathcal{U^{'}}}})^\bot$,
 and then the two cotorsion pairs coincide.

First, we prove $\mathcal{N}_{\mathcal{V^{''}}}^{\mathcal{V^{'}}}\subseteq
 ({\mathcal{M}_{\mathcal{U^{''}}}^{\mathcal{U^{'}}}})^\bot$.
For any $ X\in \mathcal{N}_{\mathcal{V^{''}}}^{\mathcal{V^{'}}}$, we have
$i^!X\in \mathcal{V'}$, $j^*X\in \mathcal{V''}$ and $(\mathbb{R}_1i^!)X=0$.
Then it follows from
Lemma~\ref{P} and Lemma~\ref{3.4} that $\delta_X $ is an epimorphism.
By Remark~\ref{2.4} (3), we have a short exact sequence $$\xymatrix@C=2pc{ 0\ar[r] & i_*i^!X\ar[r] & X\ar[r]
^{\delta_X}& j_*j^*X\ar[r]&0}.$$
Therefore, to show $ X\in ({\mathcal{M}_{\mathcal{U^{''}}}^{\mathcal{U^{'}}}})^\bot$, it suffices to
prove that $i_*i^!X \in ({\mathcal{M}_{\mathcal{U^{''}}}^{\mathcal{U^{'}}}})^\bot$ and
$j_*j^*X \in ({\mathcal{M}_{\mathcal{U^{''}}}^{\mathcal{U^{'}}}})^\bot$.
For any $Y\in \mathcal{M} _{\mathcal{U^{''}}}^{\mathcal{U^{'}}}$,
      we have $i^*Y\in {\mathcal{U^{'}}}$, $j^*Y\in{\mathcal{U}^{''}}$ and $(\mathbb{L}_1i^*)(Y)=0$.
So, it follows from Lemma~\ref{3.1} (2) that $\Ext^1_{\mathcal{A}}(Y,i_*i^!X)\cong \Ext^1_{\mathcal{A^{'}}}(i^*Y,i^!X)$,
which is zero since $i^*Y\in {\mathcal{U^{'}}}$ and $i^!X\in \mathcal{V'}$.
Therefore, we infer that $i_*i^!X \in ({\mathcal{M}_{\mathcal{U^{''}}}^{\mathcal{U^{'}}}})^\bot$.
On the other hand, it follows from
$(\mathbb{R}_1j_*)({\mathcal{V}^{''}})=0$ that $(\mathbb{R}_1j_*)(j^*X)=0$.
Applying Lemma~\ref{3.1}(3),
we have $\Ext^1_{\mathcal{A}}(Y,j_*j^*X)\cong \Ext^1_{\mathcal{A^{''}}}(j^*Y,j^*X)$,
which is zero since $j^*Y\in{\mathcal{U}^{''}}$ and $j^*X\in \mathcal{V''}$.
Therefore, we get $j_*j^*X\in ({\mathcal{M} _{\mathcal{U^{''}}}^{\mathcal{U^{'}}}})^\bot$ and then $X\in ({\mathcal{M} _{\mathcal{U^{''}}}^{\mathcal{U^{'}}}})^\bot$.

Next, we prove $({\mathcal{M} _{\mathcal{U^{''}}}^{\mathcal{U^{'}}}})^\bot \subseteq\mathcal{N}_{\mathcal{V^{''}}}^{\mathcal{V^{'}}}$.
Since $(\mathcal{U^{'}},\mathcal{V^{'}})$ and $(\mathcal{U^{''}},\mathcal{V{''}})$ are
cotorsion pairs, it follows that all of $\mathcal{U'}$, $\mathcal{V'}$, $\mathcal{U''}$ and $\mathcal{V''}$
 contain the zero object.
By Lemma~\ref{chang}, we have $i_*\mathcal{U'}\subseteq \mathcal{M}_{\mathcal{U^{''}}}^{\mathcal{U^{'}}}$
and $j_!\mathcal{U''}\subseteq \mathcal{M}_{\mathcal{U^{''}}}^{\mathcal{U^{'}}}$, and then
$\mathcal{C}_{{\mathcal{U}^{''}}}^{{\mathcal{U^{'}}}}=i_*( {\mathcal{U^{'}}})\cup{j_!({\mathcal{U}^{''}})}
\subseteq \mathcal{M} _{\mathcal{U^{''}}}^{\mathcal{U^{'}}}$.
Therefore, we get $({\mathcal{M} _{\mathcal{U^{''}}}^{\mathcal{U^{'}}}})^\bot \subseteq ({\mathcal{C}_{{\mathcal{U}^{''}}}^{\mathcal{U^{'}}}})^\perp$,
and by Lemma~\ref{lem-syz} (1), we infer that $({\mathcal{M} _{\mathcal{U^{''}}}^{\mathcal{U^{'}}}})^\bot \subseteq\mathcal{N}_{\mathcal{V^{''}}}^{\mathcal{V^{'}}}$.
\end{proof}

Motivated by \cite[Theorem 4.4]{ML}, we consider the following theorem.

\begin{theorem}\label{3.6}
Let $(\mathcal{A^{'}}, \mathcal{A}, \mathcal{A^{''}}, i^\ast, i_\ast, i^!, j_!, j^\ast, j_\ast)$
be a recollement satisfying condition (P). Assume that $\mathcal{U'}$, $\mathcal{V'}$, $\mathcal{U''}$ and $\mathcal{V''}$ are
classes of objects in $\mathcal{A'}$ and $\mathcal{A''}$,
 respectively. If
$(\mathbb{L}_1j_!)(^{\bot}{\mathcal{V^{''}}})=0$, $(\mathbb{R}_1j_*)(\mathcal{U''}^\bot)=0$, and either
$(\mathbb{L}_1j_!)(\mathcal{U^{''}})=0$ or
$(\mathbb{R}_1j_*)(\mathcal{V}^{''})=0$,
then $(\mathcal{U^{'}},\mathcal{V^{'}})$ and $(\mathcal{U^{''}},\mathcal{V{''}})$ are cotorsion pairs if and only if $(\mathcal{M}_{\mathcal{U^{''}}}^{\mathcal{U^{'}}},\mathcal{N}_{\mathcal{V^{''}}}^{\mathcal{V^{'}}})$ is a cotorsion pair.
\end{theorem}

\begin{proof}
If $(\mathcal{U^{'}},\mathcal{V^{'}})$ and $(\mathcal{U^{''}},\mathcal{V{''}})$ are cotorsion pairs, then
$\mathcal{U^{''}}= {^{\bot}{\mathcal{V^{''}}}}$ and $\mathcal{V{''}}= \mathcal{U''}^\bot$, and thus
$(\mathbb{L}_1j_!)(\mathcal{U^{''}})=0$ and
$(\mathbb{R}_1j_*)(\mathcal{V}^{''})=0$.
Then it follows from Theorem~\ref{3.3} and Proposition~\ref{prop-two-contor-equal} that $(\mathcal{M}_{\mathcal{U^{''}}}^{\mathcal{U^{'}}},\mathcal{N}_{\mathcal{V^{''}}}^{\mathcal{V^{'}}})$ is a cotorsion pair.

Conversely, assume that $(\mathcal{M}_{\mathcal{U^{''}}}^{\mathcal{U^{'}}},\mathcal{N}_{\mathcal{V^{''}}}^{\mathcal{V^{'}}})$ is a cotorsion pair in $\mathcal{A}$.
Then both $\mathcal{M}_{\mathcal{U^{''}}}^{\mathcal{U^{'}}}$ and $\mathcal{N}_{\mathcal{V^{''}}}^{\mathcal{V^{'}}}$
contain the zero object, and thus $O_{A'}=i^*O_{A}\in \mathcal{U'}$ by the definition of $\mathcal{M}_{\mathcal{U^{''}}}^{\mathcal{U^{'}}}$.
Similarly, we get that all of $\mathcal{V'}$, $\mathcal{U''}$ and $\mathcal{V''}$
contain the zero object.
For any $ A\in {\mathcal{U^{'}}}$ and $ B\in \mathcal{V^{'}}$, it follows from
\cite[Remark 3.7 (1)]{Psa14} that
$\Ext^1_{\mathcal{A^{'}}}(A,B)\cong \Ext^1_{\mathcal{A}}(i_*A,i_*B)$,
which is zero since $i_*A\in \mathcal{M}_{\mathcal{U^{''}}}^{\mathcal{U^{'}}}$ and $i_*B\in \mathcal{N}_{\mathcal{V^{''}}}^{\mathcal{V^{'}}}$,
see Lemma~\ref{chang}.
Therefore, we get ${\mathcal{U^{'}}}\subseteq {^{\perp}\mathcal{V^{'}}}$ and $\mathcal{V^{'}}\subseteq {\mathcal{U^{'}}}^\perp$.

Now we prove ${^{\perp}\mathcal{V^{'}}}\subseteq {\mathcal{U^{'}}}$.
For any $ X\in {^{\perp}\mathcal{V^{'}}}$, we will claim $i_*X\in \mathcal{M} _{\mathcal{U^{''}}}^{\mathcal{U^{'}}}$,
and then $X\cong i^*i_*X\in {\mathcal{U^{'}}}$ by the definition of $\mathcal{M} _{\mathcal{U^{''}}}^{\mathcal{U^{'}}}$.
Since $(\mathcal{M}_{\mathcal{U^{''}}}^{\mathcal{U^{'}}},\mathcal{N}_{\mathcal{V^{''}}}^{\mathcal{V^{'}}})$ is a cotorsion pair,
to show $i_*X\in \mathcal{M} _{\mathcal{U^{''}}}^{\mathcal{U^{'}}}$, it suffices to prove that
$i_*X\in {^{\perp}\mathcal{N}_{\mathcal{V^{''}}}^{\mathcal{V^{'}}}}$.
For any $Y\in \mathcal{N}_{\mathcal{V^{''}}}^{\mathcal{V^{'}}}$, we have $i^!Y\in \mathcal{V{'}}$ and $(\mathbb{R}_1i^!)Y=0$.
Then it follows from Lemma~\ref{3.1} (1) that $\Ext^1_{\mathcal{A}}(i_*X,Y)
\cong\Ext^1_{\mathcal{A^{'}}}(X,i^!Y)$, which is zero since $ X\in {^{\perp}\mathcal{V^{'}}}$.
Therefore, we get $i_*X\in {^{\perp}\mathcal{N}_{\mathcal{V^{''}}}^{\mathcal{V^{'}}}}$, and then
we obtain ${^{\perp}\mathcal{V^{'}}}\subseteq {\mathcal{U^{'}}}$.
By a similar argument, we can prove ${\mathcal{U^{'}}}^\perp \subseteq \mathcal{V^{'}}$,
and then $(\mathcal{U^{'}},\mathcal{V^{'}})$ is a cotorsion pair in $\mathcal{A'}$.

For any $ C\in {\mathcal{U}^{''}}$ and $ D\in {\mathcal{V}^{''}}$, we claim that
$\Ext^1_{\mathcal{A^{''}}}(C,D)=0$. Indeed, if $(\mathbb{L}_1j_!)(\mathcal{U^{''}})=0$,
then we have $$\Ext^1_{\mathcal{A^{''}}}(C,D)\cong \Ext^1_{\mathcal{A^{''}}}(C,j^*j_*D)\cong \Ext^1_{\mathcal{A}}(j_!C,j_*D),$$
where the second isomorphism follows from Lemma~\ref{3.1} (4).
If $(\mathbb{R}_1j_*)(\mathcal{V^{''}})=0$,
then we have $$\Ext^1_{\mathcal{A^{''}}}(C,D)\cong \Ext^1_{\mathcal{A^{''}}}(j^*j_!C,D)\cong \Ext^1_{\mathcal{A}}(j_!C,j_*D),$$
where the second isomorphism follows from Lemma~\ref{3.1} (3).
By Lemma~\ref{chang}, we get $j_!C\in\mathcal{M}_{\mathcal{U^{''}}}^{\mathcal{U^{'}}}$ and $j_*D\in \mathcal{N}_{\mathcal{V^{''}}}^{\mathcal{V^{'}}}$,
and then $\Ext^1_{\mathcal{A}}(j_!C,j_*D)=0$. Therefore,
we obtain $\Ext^1_{\mathcal{A^{''}}}(C,D)=0$ and thus
${\mathcal{U}^{''}}\subseteq {^{\perp}{\mathcal{V}^{''}}}$ and $ {\mathcal{V}^{''}}\subseteq{\mathcal{U}^{''}}^\perp$.

Now we prove ${^{\perp}\mathcal{V^{''}}}\subseteq {\mathcal{U^{''}}}$.
For any $ X\in {^{\perp}\mathcal{V^{''}}}$, we will claim $j_!X\in \mathcal{M} _{\mathcal{U^{''}}}^{\mathcal{U^{'}}}$,
and then $X\cong j^*j_!X\in {\mathcal{U^{''}}}$ by the definition of $\mathcal{M} _{\mathcal{U^{''}}}^{\mathcal{U^{'}}}$.
Since $(\mathcal{M}_{\mathcal{U^{''}}}^{\mathcal{U^{'}}},\mathcal{N}_{\mathcal{V^{''}}}^{\mathcal{V^{'}}})$ is a cotorsion pair,
to show $j_!X\in \mathcal{M} _{\mathcal{U^{''}}}^{\mathcal{U^{'}}}$, it suffices to prove that
$j_!X\in {^{\perp}\mathcal{N}_{\mathcal{V^{''}}}^{\mathcal{V^{'}}}}$.
For any $Y\in \mathcal{N}_{\mathcal{V^{''}}}^{\mathcal{V^{'}}}$, we have $j^*Y\in \mathcal{V{''}}$.
Since $(\mathbb{L}_1j_!)(^{\bot}{\mathcal{V^{''}}})=0$,
it follows from Lemma~\ref{3.1} (4) that $\Ext^1_{\mathcal{A}}(j_!X,Y)\cong
\Ext^1_{\mathcal{A^{''}}}(X,j^*Y)$, which is zero since $ X\in {^{\perp}\mathcal{V^{''}}}$.
Therefore, we get $j_!X\in {^{\perp}\mathcal{N}_{\mathcal{V^{''}}}^{\mathcal{V^{'}}}}$, and then
we obtain ${^{\perp}\mathcal{V^{''}}}\subseteq {\mathcal{U^{''}}}$.
By a similar argument, we can prove ${\mathcal{U^{''}}}^\perp \subseteq \mathcal{V^{''}}$ (where
the assumption $(\mathbb{R}_1j_*)(\mathcal{U''}^\bot)=0$ is used),
and then $(\mathcal{U^{''}},\mathcal{V^{''}})$ is a cotorsion pair in $\mathcal{A'}$.
\end{proof}

Suppose that $(\mathcal{U^{'}},\mathcal{V^{'}})$ and $(\mathcal{U^{''}},\mathcal{V{''}})$ are hereditary
or complete,
it is natural to ask if the induced cotorsion pair $(\mathcal{M}_{\mathcal{U^{''}}}^{\mathcal{U^{'}}},
\mathcal{N}_{\mathcal{V^{''}}}^{\mathcal{V^{'}}})$ have the same property.

\begin{theorem}\label{3.7}
Adopt the notations and assumptions from Theorem~\ref{3.6}. Then the cotorsion pairs
 $(\mathcal{U^{'}},\mathcal{V^{'}})$ and $(\mathcal{U^{''}},\mathcal{V{''}})$ are hereditary if and only if so is $(\mathcal{M}_{\mathcal{U^{''}}}^{\mathcal{U^{'}}},\mathcal{N}_{\mathcal{V^{''}}}^{\mathcal{V^{'}}})$.
	\end{theorem}	
		\begin{proof}
If $(\mathcal{U^{'}},\mathcal{V^{'}})$ and $(\mathcal{U^{''}},\mathcal{V{''}})$ are hereditary cotorsion pairs,
then it follows from Theorem~\ref{3.6} that $(\mathcal{M}_{\mathcal{U^{''}}}^{\mathcal{U^{'}}},
\mathcal{N}_{\mathcal{V^{''}}}^{\mathcal{V^{'}}})$ is a cotorsion pair. It remains to show that
$(\mathcal{M}_{\mathcal{U^{''}}}^{\mathcal{U^{'}}},
\mathcal{N}_{\mathcal{V^{''}}}^{\mathcal{V^{'}}})$ is hereditary. In view of
Proposition~\ref{he1}, it suffices to prove that $\mathcal{M}_{\mathcal{U^{''}}}^{\mathcal{U^{'}}}$ is closed under the kernel of epimorphisms.
Let $\xymatrix@C=1.5pc{0\ar[r]& X\ar[r]^{f} & Y\ar[r]^g &Z \ar[r]& 0}$ be a short exact sequence in $\mathcal{A}$
with $Y,Z\in\mathcal{M}_{\mathcal{U^{''}}}^{\mathcal{U^{'}}}$.
Then we have $i^*Y\in\mathcal{U{'}}$, $i^*Z\in\mathcal{U{'}}$, $j^*Y\in{\mathcal{U}^{''}}$, $j^*Z\in{\mathcal{U}^{''}}$,
$(\mathbb{L}_1i^*)(Y)=0$ and $(\mathbb{L}_1i^*)(Z)=0$. Therefore,
we get two exact sequences $$\xymatrix@C=2pc{0\ar[r]& i^*X\ar[r]^{i^*f} & i^*Y\ar[r]^{i^*g} &i^*Z \ar[r]& 0}$$ and $$\xymatrix@C=2pc{0\ar[r]& j^*X\ar[r]^{j^*f} & j^*Y\ar[r]^{j^*g} &j^*Z\ar[r]& 0},$$
which implies $i^*X\in\mathcal{U{'}}$ and $j^*X\in{\mathcal{U}^{''}}$
since $(\mathcal{U^{'}},\mathcal{V^{'}})$ and $(\mathcal{U^{''}},\mathcal{V{''}})$ are hereditary. Therefore,
 to show $X\in\mathcal{M}_{\mathcal{U^{''}}}^{\mathcal{U^{'}}}$, it remains to prove $(\mathbb{L}_1i^*)(X)=0$.
Since $j^*Z\in{\mathcal{U}^{''}}= {^{\bot}{\mathcal{V^{''}}}}$
and $(\mathbb{L}_1j_!)(^{\bot}{\mathcal{V^{''}}})=0$, we get $(\mathbb{L}_1j_!)(j^*Z)=0$, and then
we have the following short exact sequence
 $$\xymatrix@C=2pc{0\ar[r]& j_!j^*X\ar[r]^{j_!j^*f} & j_!j^*Y\ar[r]^{j_!j^*g} &j_!j^*Z \ar[r]& 0}.$$
On the other hand, it follows from Lemma~\ref{3.4} that
$\varepsilon _Y$ is a monomorphism, and so is $\varepsilon_X$
by the following commutative diagram
 $$\xymatrix@C=2pc{
			0\ar[r]& j_!j^*X \ar[r]^{j_!j^*f}\ar[d]^{\varepsilon_X}& j_!j
^*Y\ar[r]\ar[d]^{\varepsilon_Y} & j_!j^*Z \ar[r]\ar[d]^{\varepsilon_Z}  & 0\\0\ar[r]&X\ar[r]^{f}&Y\ar[r]&Z\ar[r]&0.}$$
Applying Lemma~\ref{3.4} again, we obtain that $(\mathbb{L}_1i^*)(X)=0$ and thus $X\in\mathcal{M}_{\mathcal{U^{''}}}^{\mathcal{U^{'}}} $.

Conversely, assume that $(\mathcal{M}_{\mathcal{U^{''}}}^{\mathcal{U^{'}}},\mathcal{N}_{\mathcal{V^{''}}}^{\mathcal{V^{'}}})$
is a hereditary cotorsion pair. Then it follows from Theorem~\ref{3.6} that
both $(\mathcal{U^{'}},\mathcal{V^{'}})$ and $(\mathcal{U^{''}},\mathcal{V{''}})$ are cotorsion pairs.
It remains to show that both
$(\mathcal{U^{'}},\mathcal{V^{'}})$ and $(\mathcal{U^{''}},\mathcal{V{''}})$ are hereditary. In view of
Proposition~\ref{he1}, it suffices to prove that
both $\mathcal{V{'}}$ and ${\mathcal{V}^{''}}$ are closed under the cokernel of monomorphisms.

Let $\xymatrix@C=1.5pc{0\ar[r]& X'\ar[r]^f & Y'\ar[r]^g &Z' \ar[r]& 0}$ be a short exact sequence in
$\mathcal{A}'$ with $X',Y'\in\mathcal{V{'}}$ .
Applying $i_*$, we get a short exact sequence
  $$\xymatrix@C=2pc{0\ar[r]& i_*X'\ar[r]^{i_*f} & i_*Y'\ar[r]^{i_*f} &i_*Z' \ar[r]& 0}.$$
On the other hand, it follows from Lemma~\ref{chang} that $i_*X'$, $i_*Y'\in\mathcal{N}_{\mathcal{V^{''}}}^{\mathcal{V^{'}}}$,
which implies $i_*Z'\in\mathcal{N}_{\mathcal{V^{''}}}^{\mathcal{V^{'}}}$
since $(\mathcal{M}_{\mathcal{U^{''}}}^{\mathcal{U^{'}}},\mathcal{N}_{\mathcal{V^{''}}}^{\mathcal{V^{'}}})$ is hereditary.
Therefore, we get $Z'\cong i^!i_*Z'\in\mathcal{V{'}}$ by the definition of $\mathcal{N}_{\mathcal{V^{''}}}^{\mathcal{V^{'}}}$.

Let $\xymatrix@C=1.5pc{0\ar[r]& X^{''}\ar[r]^\alpha & Y^{''}\ar[r]^\beta &Z^{''} \ar[r]& 0}$ be a short exact sequence in
$\mathcal{A}''$ with
$X^{''},Y^{''}\in{\mathcal{V}^{''}}$. Since ${\mathcal{V}^{''}}=\mathcal{U''}^\bot$ and
$(\mathbb{R}_1j_*)(\mathcal{U''}^\bot)=0$,
we get a short exact sequence $\xymatrix@C=2pc{0\ar[r]& j_*X^{''}\ar[r]^{j_*(\alpha)} & j_*Y^{''}\ar[r]^{j_*(\beta)} &j_*Z^{''} \ar[r]& 0}$.
On the other hand, it follows from Lemma~\ref{chang} that
$j_*X^{''}$, $j_*Y^{''}\in\mathcal{N}_{\mathcal{V^{''}}}^{\mathcal{V^{'}}}$,
which implies $j_*Z^{''}\in\mathcal{N}_{\mathcal{V^{''}}}^{\mathcal{V^{'}}}$
since $(\mathcal{M}_{\mathcal{U^{''}}}^{\mathcal{U^{'}}},\mathcal{N}_{\mathcal{V^{''}}}^{\mathcal{V^{'}}})$ is hereditary.
Therefore, we get $Z^{''}\cong j^*j_*Z^{''}\in{\mathcal{V}^{''}}$ by the definition of $\mathcal{N}_{\mathcal{V^{''}}}^{\mathcal{V^{'}}}$.
		\end{proof}

To verify the completeness of the cotorsion pairs constructed earlier,
we need the following two lemmas. Recall that a subcategory is {\it resolving}
if it is closed under extensions and
kernels of epimorphisms and contains all projective
objects \cite{GR}. Let $\mathcal{C}$ be a class of objects of $\mathcal{A}$.
For any $X\in \mathcal{A}$, we say $X$ has a {\it special $\mathcal{C}$-precover}
if there is a short exact sequence $0\rightarrow Y\rightarrow C \rightarrow X \rightarrow0$ with $C\in\mathcal{C}$ and $Y\in{\mathcal{C}^\perp}$
\cite{EEE}.
Dually we have
the definitions of coresolving subcategory and special $\mathcal{C}$-preenvelope.
In the following two lemmas, we assume that $\mathcal{C}$ and $\mathcal{D}$ are two
subcategories of $\mathcal{A}$ such that $\mathcal{C}$ is resolving and $\mathcal{D}$ is coresolving.

\begin{lemma}{\rm \cite[Lemma 5.4]{ML}}\label{yong1}
Let $\xymatrix@C=1.5pc{0\ar[r]& A_1\ar[r] & A_2\ar[r] &A_3 \ar[r]& 0}$ be a
short exact sequence in $\mathcal{A}$.

{\rm(1)} If $A_1$ and $A_3$ have special $\mathcal{C}$-precovers, then so does $A_2$.

{\rm(2)} If $A_1$ and $A_3$ have special $\mathcal{D}$-preenvelopes, then so does $A_2$.
\end{lemma}

\begin{lemma}\label{yong2}
Let $\xymatrix@C=1.5pc{0\ar[r]& A_1\ar[r]^f & A_2\ar[r]^g &A_3 \ar[r]& 0}$ be a short exact sequence
in $\mathcal{A}$.

{\rm(1)} If $A_1\in\mathcal{C}^{\bot}$ and $A_2$ has a special $\mathcal{C}$-precover, then so does $A_3$.

{\rm(2)} If $A_1\in \mathcal{D}$ and $A_2$ has a special $\mathcal{D}$-preenvelope, then so does $A_3$.

{\rm(3)} If $A_3\in \mathcal{C}$ and $A_2$ has a special $\mathcal{C}$-precover, then so does $A_1$.

{\rm(4)} If $A_3\in {^{\bot}\mathcal{D}}$ and $A_2$ has a special $\mathcal{D}$-preenvelope, then so does $A_1$ .
\end{lemma}

\begin{proof}
  We only prove (1) and (2), and the proof of (3) and (4) are similar.

  (1) Since $A_2$ has a special $\mathcal{C}$-precover,
there is a short exact sequence $\xymatrix@C=1.5pc{0\ar[r]& B_3\ar[r]^l & C_2\ar[r]^\rho &A_2 \ar[r]& 0}$ with $C_2\in \mathcal{C}$ and $B_3\in \mathcal{C}^{\bot}$. Then we have the following pullback diagram:
  $$\xymatrix@C=1.5pc{ &0\ar[d]&0\ar[d]&&\\  & B_3\ar[d]\ar@{=}[r] &B_3
			\ar[d] &  &  \\
			0\ar[r]& K\ar[r]\ar[d]& C_2\ar[r]^\sigma\ar[d]^\rho & A_3\ar[r]\ar@{=}[d]  & 0\\0\ar[r]&A_1\ar[r]^f\ar[d]&A_2\ar[r]^g\ar[d]&A_3\ar[r]&0.\\&0&0&&}$$
Since $\mathcal{C}^{\bot}$ is closed under extensions, we get $K\in\mathcal{C}^{\bot}$ by the
first column. Now the second row yields that $\sigma$ is a special $\mathcal{C}$-precover of $A_3$.

(2) Since $A_2$ has a special $\mathcal{D}$-preenvelope,
there is a short exact sequence $\xymatrix@C=1.5pc{0\ar[r]& A_2\ar[r]^i & D_2\ar[r]^\pi &B_2 \ar[r]& 0}$ with $D_2\in \mathcal{D}$ and $B_2\in {^{\bot}\mathcal{D}}$.
Then we have the following pushout diagram:
  $$\xymatrix@C=1.5pc{ &&0\ar[d]&0\ar[d]&\\ 0\ar[r] & A_1\ar@{=}[d]\ar[r]^f &A_2\ar[r]^g
			\ar[d]^i & A_3\ar[d]^\tau\ar[r] &0  \\
			0\ar[r]& A_1\ar[r]& D_2\ar[r]\ar[d] & H\ar[r]\ar[d]  & 0.\\&&B_2\ar@{=}[r]\ar[d]&B_2\ar[d]&\\&&0&0&}$$
  Since $\mathcal{D}$ is closed under cokernels of monomorphisms, we get $H\in \mathcal{D}$ from the
  second row. Now the third column yields that $\tau$ is a special $\mathcal{D}$-preenvelope of $A_3$.
\end{proof}

Now we are ready to verify the completeness of $(\mathcal{M}_{\mathcal{U^{''}}}^{\mathcal{U^{'}}}, \mathcal{N}_{\mathcal{V^{''}}}^{\mathcal{V^{'}}})$.
\begin{theorem}\label{3.8(1)}
Let $(\mathcal{A^{'}}, \mathcal{A}, \mathcal{A^{''}}, i^\ast, i_\ast, i^!, j_!, j^\ast, j_\ast)$
be a recollement satisfying condition (P). Assume that $({\mathcal{U^{'}}},\mathcal{V{'}})$ and $
({\mathcal{U}^{''}},{\mathcal{V}^{''}})$ are complete hereditary cotorsion pairs in $\mathcal{A^{'}}$ and $\mathcal{A^{''}}$, respectively.
If $(\mathbb{L}_1j_!)({\mathcal{U^{''}}})=0$ and $(\mathbb{R}_1j_*)({\mathcal{V^{''}}})=0$,
then $(\mathcal{M}_{\mathcal{U^{''}}}^{\mathcal{U^{'}}}, \mathcal{N}_{\mathcal{V^{''}}}^{\mathcal{V^{'}}})$ is a complete hereditary cotorsion pair in $\mathcal{A}$.
\end{theorem}

\begin{proof}
Since $({\mathcal{U}^{''}},{\mathcal{V}^{''}})$ is a cotorsion pair, we get $^\perp {\mathcal{V''}}=\mathcal{U''}$
and ${\mathcal{U''}}^\bot=\mathcal{V''}$.
Then we have $(\mathbb{L}_1j_!)(^\perp {\mathcal{V''}})=0$ and $(\mathbb{R}_1j_*)({\mathcal{U''}}^\bot)=0$,
and by Theorem~\ref{3.7}, we get that $(\mathcal{M}_{\mathcal{U^{''}}}^{\mathcal{U^{'}}}, \mathcal{N}_{\mathcal{V^{''}}}^{\mathcal{V^{'}}})$ is a hereditary cotorsion pair, which yields that $\mathcal{M}_{\mathcal{U^{''}}}^{\mathcal{U^{'}}}$ is resolving and $\mathcal{N}_{\mathcal{V^{''}}}^{\mathcal{V^{'}}}$ is coresolving,
see Proposition~\ref{he1}.
Now we prove the completeness of $(\mathcal{M}_{\mathcal{U^{''}}}^{\mathcal{U^{'}}}, \mathcal{N}_{\mathcal{V^{''}}}^{\mathcal{V^{'}}})$
by the following three steps.

Step 1:
For any $Y\in\mathcal{A}'$, we claim that $i_*Y$ has a special $\mathcal{M}_{\mathcal{U^{''}}}^{\mathcal{U^{'}}}$-precover
and a special $\mathcal{N}_{\mathcal{V^{''}}}^{\mathcal{V^{'}}}$-preenvelope.
Since $(\mathcal{U^{'}},\mathcal{V^{'}})$ is complete, there is a short exact sequence
$\xymatrix@C=2pc{0\ar[r]& V'_1\ar[r] &U'_1\ar[r] &Y \ar[r]& 0}$
with $V'_1\in\mathcal{V{'}}$ and $U'_1\in{\mathcal{U^{'}}}$. Applying $i_*$,
 we have the following
exact sequence
   \begin{center}
   $\xymatrix@C=2pc{0\ar[r]& i_*V'_1\ar[r] &i_*U'_1\ar[r] &i_*Y \ar[r]& 0}$,
   \end{center}
 which yields that $i_*Y$ has a special $\mathcal{M}_{\mathcal{U^{''}}}^{\mathcal{U^{'}}}$-precover
 since $i_*V'_1\in\mathcal{N}_{\mathcal{V^{''}}}^{\mathcal{V^{'}}}$ and $i_*U'_1\in\mathcal{M}_{\mathcal{U^{''}}}^{\mathcal{U^{'}}}$,
see Lemma~\ref{chang}. Dually, we can prove that $i_*Y$ has a special $\mathcal{N}_{\mathcal{V^{''}}}^{\mathcal{V^{'}}}$-preenvelope.

Step 2:
For any $X\in\mathcal{A}$, we claim that $j_*j^*X$ has a special $\mathcal{M}_{\mathcal{U^{''}}}^{\mathcal{U^{'}}}$-precover.
Indeed, since $j^*X\in\mathcal{A''}$ and $(\mathcal{U^{''}},\mathcal{V{''}})$ is complete, we have a short exact sequence
$\xymatrix@C=1.5pc{0\ar[r]& V''_1\ar[r] & U''_1\ar[r]^f &j^*X \ar[r]& 0}$  with $V''_1\in{\mathcal{V}^{''}}$ and $U''_1\in{\mathcal{U}^{''}}$.
Since $(\mathbb{R}_1j_*)({\mathcal{V{''}}})=0$, we get a short exact sequence
$$\xymatrix@C=2pc{0\ar[r]& j_*V''_1\ar[r] &j_* U''_1\ar[r]^{j_*f} &j_*j^*X \ar[r]& 0}.$$
On the other hand, it follows from Lemma~\ref{chang} that $j_*\mathcal{V{''}}\in \mathcal{N}_{\mathcal{V^{''}}}^{\mathcal{V^{'}}}$.
To show $j_*j^*X$ has a special $\mathcal{M}_{\mathcal{U^{''}}}^{\mathcal{U^{'}}}$-precover,
in view of Lemma~\ref{yong2} (1), it suffices to show that $j_*U''_1$ has a special $\mathcal{M}_{\mathcal{U^{''}}}^{\mathcal{U^{'}}}$-precover.
By Remark~\ref{2.4} (3),
there is an exact sequence
$$\xymatrix@C=2pc{0\ar[r]& i_*X'\ar[r] & j_!j^*j_*U''_1\ar[rr]^{\varepsilon _{j_* U''_1}} &&j_* U''_1\ar[r] &i_*i^*j_*U''_1 \ar[r]& 0}$$
with $X'\in \mathcal{A'}$.
Since $j_!j^*j_*U''_1\cong j_!U''_1$, we get an exact sequence
$$\xymatrix@C=2pc{0\ar[r]& i_*X'\ar[r]^f & j_!U''_1\ar[r] &j_* U''_1\ar[r] &i_*i^*j_*U''_1 \ar[r]& 0},$$
which yields two short exact sequences
$\xymatrix@C=1.5pc{0\ar[r]& i_*X'\ar[r]^f & j_!U''_1\ar[r]&L\ar[r] &0}$
    and $\xymatrix@C=1.5pc{0\ar[r]&L\ar[r]& j_* U''_1\ar[r] &i_*i^*j_*U''_1 \ar[r]& 0}$
with $L=\Coker f$. It follows from step 1 that $i_*X'$ has a special $\mathcal{N}_{\mathcal{V^{''}}}^{\mathcal{V^{'}}}$-preenvelope,
that is, there is a short exact sequence $\xymatrix@C=1.5pc{0\ar[r]& i_*X'\ar[r]^{\alpha} & N\ar[r]&M\ar[r] &0}$
with $M\in \mathcal{M}_{\mathcal{U^{''}}}^{\mathcal{U^{'}}}$ and $N\in \mathcal{N}_{\mathcal{V^{''}}}^{\mathcal{V^{'}}}$.
Then we have
the following pushout diagram:
$$\xymatrix@C=2pc{ & 0\ar[d] & 0\ar[d] & &\\
0\ar[r]& i_*X'\ar[r]^f \ar[d]_{\alpha}& j_!U''_1\ar[r]\ar[d]&L\ar[r]\ar@{=}[d] &0 \\
0\ar[r]& N\ar[r]\ar[d] & W\ar[d]\ar[r]&L\ar[r] &0.\\
& M\ar[d] \ar@{=}[r]& M\ar[d]& & \\
 & 0 & 0 & &
} $$
By Lemma~\ref{chang}, we get $j_!U''_1\in \mathcal{M}_{\mathcal{U^{''}}}^{\mathcal{U^{'}}}$,
and then we deduce $W\in \mathcal{M}_{\mathcal{U^{''}}}^{\mathcal{U^{'}}}$ from the second
column. Now
the second row yields that $L$ has a special $\mathcal{M}_{\mathcal{U^{''}}}^{\mathcal{U^{'}}}$-precover.
By step 1, we have that $i_*i^*j_*U''_1$ has a special $\mathcal{M}_{\mathcal{U^{''}}}^{\mathcal{U^{'}}}$-precover.
Applying Lemma~\ref{yong1} (1) to
$\xymatrix@C=1.5pc{0\ar[r]&L\ar[r]& j_* U''_1\ar[r] &i_*i^*j_*U''_1 \ar[r]& 0}$, we
get that $j_*U''_1$ has a special $\mathcal{M}_{\mathcal{U^{''}}}^{\mathcal{U^{'}}}$-precover.

   Step 3: For any $X\in\mathcal{A}$, it follows from Remark~\ref{2.4} (3) that there is an exact sequence
   \begin{center}
$\xymatrix@C=2pc{0\ar[r]& i_*i^!X\ar[r]^{\eta_X}& X\ar[r]^{\delta _X} & j_*j^*X\ar[r]^g &i_*Y' \ar[r]& 0}$
\end{center}
 with $Y'\in\mathcal{A'}$. Let $\Coker \eta_X=K$. Then there are two exact sequences
$\xymatrix@C=1.5pc{0\ar[r]& i_*i^!X\ar[r]^{\delta_X}& X\ar[r] &K\ar[r]& 0}$
    and $\xymatrix@C=1.5pc{0\ar[r]&K\ar[r]& j_*j^*X\ar[r]^g &i_*Y' \ar[r]& 0}$.
Since $Y'\in\mathcal{A'}$ and $(\mathcal{U^{'}},\mathcal{V^{'}})$ is complete, there is a short exact suquence
     \begin{center}
   $\xymatrix@C=2pc{0\ar[r]&V'_2\ar[r]& U'_2\ar[r]^{\beta} &Y' \ar[r]& 0}$
    \end{center}
    with $V'_2\in\mathcal{V'} $ and $U'_2\in\mathcal{U'} $. Applying $i_*$, we get the following
    exact sequence $\xymatrix@C=2pc{0\ar[r]&i_*V'_2\ar[r]& i_*U'_2\ar[r]^{i_*\beta} &i_*Y' \ar[r]& 0}$, and then
 we have pullback diagram:
  $$\xymatrix@C=2pc{ &&0\ar[d]  &0\ar[d]&\\ & & i_*V'_2 \ar[d]\ar@{=}[r] &i_*V'_2
			\ar[d] &  \\
			0\ar[r]& K\ar[r]\ar@{=}[d]& Z\ar[r]\ar[d] & i_*U'_2\ar[r]\ar[d]^{i_*\beta}  & 0\\0\ar[r]&K\ar[r]&j_*j^*X\ar[r]^g\ar[d]&i_*Y'\ar[d]\ar[r]&0.\\&&0&0&}$$
By step 1 and step 2, we get that both $i_*V'_2$ and $j_*j^*X$ have special $\mathcal{M}_{\mathcal{U^{''}}}^{\mathcal{U^{'}}}$-precovers.
Applying Lemma~\ref{yong1} (1) to the second column, we deduce that $Z$ has a special $\mathcal{M}_{\mathcal{U^{''}}}^{\mathcal{U^{'}}}$-precover.
On the other hand, it follows from Lemma~\ref{chang} that $i_*U'_2\in\mathcal{M}_{\mathcal{U^{''}}}^{\mathcal{U^{'}}}$.
Applying Lemma~\ref{yong2} (3) to the second row, we get that $K$ has a special $\mathcal{M}_{\mathcal{U^{''}}}^{\mathcal{U^{'}}}$-precover.
Applying Lemma~\ref{yong1} (1) to the exact sequence $\xymatrix@C=2pc{0\ar[r]&i_*i^!X\ar[r]& X\ar[r] &K \ar[r]& 0}$, and by step 1,
we get that $X$ has a special $\mathcal{M}_{\mathcal{U^{''}}}^{\mathcal{U^{'}}}$-precover.
Now it follows from Proposition~\ref{co1} that the cotorsion pair
$(\mathcal{M}_{\mathcal{U^{''}}}^{\mathcal{U^{'}}}, \mathcal{N}_{\mathcal{V^{''}}}^{\mathcal{V^{'}}})$ is complete.
  \end{proof}

 Motivated by \cite[Theorem 5.6]{ML}, we consider the converse of Theoren~\ref{3.8(1)}.
\begin{proposition}\label{3.16}
Let $(\mathcal{A^{'}}, \mathcal{A}, \mathcal{A^{''}}, i^\ast, i_\ast, i^!, j_!, j^\ast, j_\ast)$
be a recollement satisfying condition (P), and let ${\mathcal{U^{'}}}, \mathcal{V{'}}, \ {\mathcal{U}^{''}}$ and ${\mathcal{V}^{''}}$ be classes of objects
in $\mathcal{A^{'}}$ and $\mathcal{A^{''}}$, respectively. Assume that
$(\mathcal{M}_{\mathcal{U^{''}}}^{\mathcal{U^{'}}}, \mathcal{N}_{\mathcal{V^{''}}}^{\mathcal{V^{'}}})$ is a complete hereditary cotorsion pair in $\mathcal{A}$,
$(\mathbb{R}_1j_*)({\mathcal{U^{''}}}^\bot)=0$, $(\mathbb{L}_1j_!)(^{\bot}{\mathcal{V^{''}}})=0$,
and either $(\mathbb{L}_1j_!)({\mathcal{U^{''}}})=0$ or $(\mathbb{R}_1j_*)({\mathcal{V^{''}}})=0$.
Then both $({\mathcal{U^{'}}},\mathcal{V{'}})$ and $ ({\mathcal{U}^{''}},{\mathcal{V}^{''}})$ are complete hereditary cotorsion pairs
if one of the following two conditions holds:

{\rm (1)} $i^*j_*{\mathcal{V^{''}}}\subseteq {\mathcal{V^{'}}}$ and $(\mathbb{L}_1i^*)j_*{\mathcal{V^{''}}}=0$;

{\rm (2)} $i^!j_!{\mathcal{U}^{''}}\subseteq{\mathcal{U^{'}}}$ and $(\mathbb{R}_1i^!)j_!{\mathcal{U}^{''}}=0$.
\end{proposition}

\begin{proof}
It follows from Theorem~\ref{3.6} and Theorem~\ref{3.7} that both $({\mathcal{U^{'}}},\mathcal{V{'}})$ and $ ({\mathcal{U}^{''}},{\mathcal{V}^{''}})$ are hereditary cotorsion pairs.
Now we prove $ ({\mathcal{U}^{''}},{\mathcal{V}^{''}})$ is complete.
For any $X\in \mathcal{A^{''}}$, since $j_*X\in \mathcal{A}$
and $(\mathcal{M}_{\mathcal{U^{''}}}^{\mathcal{U^{'}}}, \mathcal{N}_{\mathcal{V^{''}}}^{\mathcal{V^{'}}})$ is complete, we have
a short exact sequence
$\xymatrix@C=1.5pc{0\ar[r]& j_*X\ar[r] & N\ar[r]^f &M \ar[r]& 0}$ with $N\in\mathcal{N}_{\mathcal{V^{''}}}^{\mathcal{V^{'}}}$ and $M\in\mathcal{M}_{\mathcal{U^{''}}}^{\mathcal{U^{'}}}$. Then we get $j^*N\in \mathcal{V''}$ and $j^*M\in \mathcal{U''}$,
and we have a short exact sequence $\xymatrix@C=1.5pc{0\ar[r]& X\ar[r] & j^*N\ar[r]^{j^*f} &j^*M \ar[r]& 0}.$
This implies that the cotorsion pair $ ({\mathcal{U}^{''}},{\mathcal{V}^{''}})$ is complete.

Next we prove the completeness of $({\mathcal{U^{'}}},\mathcal{V{'}})$ under the condition (1).
For any $X\in \mathcal{A^{'}}$, since $i_*X\in \mathcal{A}$
and $(\mathcal{M}_{\mathcal{U^{''}}}^{\mathcal{U^{'}}}, \mathcal{N}_{\mathcal{V^{''}}}^{\mathcal{V^{'}}})$ is complete, we have
a short exact sequence
$\xymatrix@C=1.5pc{0\ar[r]& i_*X\ar[r] & W\ar[r]^f &Q \ar[r]& 0}$ with $W\in\mathcal{N}_{\mathcal{V^{''}}}^{\mathcal{V^{'}}}$ and $Q\in\mathcal{M}_{\mathcal{U^{''}}}^{\mathcal{U^{'}}}$. Then we have
$i^!W\in \mathcal{V'}$, $j^*W\in \mathcal{V''}$, $(\mathbb{R}_1i^!)W=0$,
$i^*Q\in \mathcal{U'}$ and $(\mathbb{L}_1i^*)Q=0$.
Applying $i^*$, we have a short exact sequence
  $\xymatrix@C=1.5pc{0\ar[r]& X\ar[r] & i^*W\ar[r]^{i^*f} &i^*Q \ar[r]& 0}$.
Since $i^*Q\in \mathcal{U'}$, to verify the completeness of $({\mathcal{U^{'}}},\mathcal{V{'}})$,
it remains to show that $i^*W\in \mathcal{V{'}}$.
On the other hand, it follows from Lemma~\ref{3.4} that $\delta _W$
is an epimorphism, and by Remark~\ref{2.4} (3),
there is a short exact sequence $$\xymatrix@C=2.5pc{0\ar[r]& i_*i^!W\ar[r] & W\ar[r]^{\delta _W} &j_*j^*W \ar[r]& 0}.$$
Since $j^*W\in \mathcal{V''}$ and $(\mathbb{L}_1i^*)j_*\mathcal{V''}=0$,
the sequence $$\xymatrix@C=2pc{0\ar[r]& i^!W\ar[r] & i^*W\ar[rr]^{i^*(\delta _W)} &&i^*j_*j^*W \ar[r]& 0}$$ is also exact.
Since $i^!W\in \mathcal{V'}$ and $i^*j_*j^*W\in i^*j_*\mathcal{V''}\subseteq \mathcal{V'}$,
we get $i^*W\in \mathcal{V{'}}$ by the above exact sequence, and then the cotorsion pair $({\mathcal{U^{'}}},\mathcal{V{'}})$ is complete.
If condition (2) holds, then the completeness of $({\mathcal{U^{'}}},\mathcal{V{'}})$ can be verified in the same manner-starting instead with the special $\mathcal{M}_{\mathcal{U^{''}}}^{\mathcal{U^{'}}}$-precover of $i_*X$ and applying $i^!$.
\end{proof}

\section{Applications and examples}\label{4}
\indent\indent  As shown in Section~\ref{Gluing cotorsion pairs along recollements}, the condition (P)
imposed on the recollement
plays a crucial role. In this section, we first provide
a sufficient condition for a recollement to satisfy the condition (P), and then we
apply our main results to Morita context rings
and triangular matrix rings.

\begin{proposition}\label{pro-exact-recoll}
Let $(\mathcal{A^{'}}, \mathcal{A}, \mathcal{A^{''}}, i^\ast, i_\ast, i^!, j_!, j^\ast, j_\ast)$ be a recollement
of abelian categories. Then it satisfies the condition (P) if $i^!$ or $i^*$ is exact.
\end{proposition}
\begin{proof}
If $i^*$ is exact, then it follows from \cite[Proposition 3.3, Lemma 3.1]{HJS} that $\varepsilon_P$ is a monomorphism for any $ P\in \mathcal{P_{\mathcal{A}}}$.
Also, if $i^!$ is exact, then $\delta_I$ is an epimorphism for any $ I\in \mathcal{I_{\mathcal{A}}}$.
Due to Lemma~\ref{P}, we get that $\varepsilon_P$ is a monomorphism for any $ P\in \mathcal{P_{\mathcal{A}}}$.
\end{proof}

 Applying Theorem~\ref{3.8(1)} and Proposition~\ref{pro-exact-recoll}, we have the following corollary.
\begin{corollary}\label{cor-ecact-recoll}
Let $(  \mathcal{A^{'}},\mathcal{A},\mathcal{A^{''}},i^\ast,i_\ast,i^!,j_!,j^\ast,j_\ast)$
be a recollement of abelian categories. Assume that $({\mathcal{U^{'}}},\mathcal{V{'}})$ and $
({\mathcal{U}^{''}},{\mathcal{V}^{''}})$ are complete hereditary cotorsion pairs in $\mathcal{A^{'}}$ and $\mathcal{A^{''}}$, respectively.
Then
   $(\mathcal{M}_{\mathcal{U^{''}}}^{\mathcal{U^{'}}}, \mathcal{N}_{\mathcal{V^{''}}}^{\mathcal{V^{'}}})$ is a complete hereditary cotorsion pair in $\mathcal{A}$
if one of the
following conditions holds:

 {\rm (1)} $i^!$ is exact and $(\mathbb{L}_1j_!)({\mathcal{U^{''}}})=0$;

 {\rm (2)} $i^*$ is exact and $(\mathbb{R}_1j_*)({\mathcal{V^{''}}})=0$;

 {\rm (3)} both $i^!$ and $i^*$ are exact.
\end{corollary}
\begin{proof}
By \cite[Lemma 3.1, Lemma 3.2]{FZ17}, the exactness of $i^!$ (resp. $i^*$) implies the exactness of $j_*$
(resp. $j_!$). Now this corollary follows from Theorem~\ref{3.8(1)} and Proposition~\ref{pro-exact-recoll}.
\end{proof}

\begin{remark}\label{rek-exact-recoll}
{ \rm (1) Corollary~\ref{cor-ecact-recoll} (1) is contained in \cite[Theorem 1.1]{HJS}, where the authors
glue complete hereditary cotorsion pairs along recollement of exact categories under the assumption
that $i^!$ is exact and $j_!$ is $\mathcal{U''}$-exact.
We mention that their condition $j_!$ is $\mathcal{U''}$-exact is equivalent to $(\mathbb{L}_1j_!)({\mathcal{U^{''}}})=0$,
and their symbols $\mathcal{M}_{\mathcal{U^{''}}}^{\mathcal{U^{'}}}$ and $\mathcal{N}_{\mathcal{V^{''}}}^{\mathcal{V^{'}}}$
coincide with ours as $i^!$ is exact here, and the counit $\varepsilon _X$ is a monomorphism precisely when $(\mathbb{L}_1i^*)(X)=0$, by Lemma~\ref{3.4}.

(2) Corollary~\ref{cor-ecact-recoll} (3) is considered in \cite[Corollary 3.3]{CW}.
}
\end{remark}

Now, we will recall some facts on Morita rings. The reader is referred to \cite{ELG} for more details.
Let $A$ and $B$ be rings, $_BM_A$ a $B$-$A$-bimodule, $_AN_B$ a $A$-$B$-bimodule, $\phi: M\otimes_AN\rightarrow B$ a $B$-bimodule map, and $\psi: N\otimes_B M\rightarrow A$ an $A$-bimodule map, such that
$$m'
   \psi(n\otimes_B
    m)=\phi(m' \otimes_A n)m,$$
     $$n' \phi(m\otimes_An)=\psi(n' \otimes_B m)n, \ \forall \ m,m'\in M, \ \forall \ n,n'\in N.$$

A {\it Morita ring} is $\Lambda=\Lambda_{(\phi,\psi)}:=\left( \begin{array} {cc}
	A	&_AN_B  \\
_BM_A	& B
	\end{array}\right)$, with componentwise addition, and multiplication \begin{center}
$\left( \begin{array} {cc}
	a	&n  \\
m	& b
	\end{array}\right)\left( \begin{array} {cc}
	a'	&n' \\
m'& b'
	\end{array}\right)=\left( \begin{array} {cc}
	aa'+\psi(n\otimes_Bm')	&an'+nb' \\
ma'+bm'	& \phi(m\otimes_An')+bb'
	\end{array}\right).$
\end{center}

A left $\Lambda$-module is identified with a quadruple $\left( \begin{array} {c}
	X\\
Y
	\end{array}\right)_{f,g}$, where $X\in A$-$\Mod$, $Y\in B$-$\Mod$, $f\in\Hom_B(M\otimes_AX,Y)$ and $g\in\Hom_A(N\otimes_BY,X)$ such that
 $$g(n\otimes_Bf(m\otimes_Ax))=\psi(n\otimes_Bm)x,$$
$$f(m\otimes_Ag(n\otimes_By))=\phi(m\otimes_An)y, \ \forall \ m\in M, \ n\in N,\  x\in X, \ y\in Y.$$
A homomorphism of $\Lambda$-module from $\left( \begin{array} {c}
	X\\
Y
	\end{array}\right)_{f,g}$ to $\left( \begin{array} {c}
	X'\\
Y'
	\end{array}\right)_{f',g'}$ is a pair
 ${\scriptsize\left( \begin{array} {c}
	a\\
b
	\end{array}\right)}$,
where $a\in \Hom _A(X,X')$ and $b\in \Hom _B(Y,Y')$ such that the following diagrams are commutative:
\noindent $$\xymatrix@C=2pc{M\otimes_AX\ar[r]^{1_M\otimes a}\ar[d]^{f}&M\otimes_AX^{'}\ar[d]^{f^{'}}
\\
Y\ar[r]^{b}&Y^{'},}\ \ \
 \xymatrix@C=2pc{N\otimes_BY\ar[r]^{1_N\otimes b}\ar[d]^{g}&N\otimes_BY^{'}\ar[d]^{g^{'}}
\\
X\ar[r]^{a}&X^{'}.}
$$

Let $\eta_{X,Y}:\Hom_B(M\otimes_AX,Y)\cong \Hom_A(X,\Hom_B(M,Y))$ and $\eta'_{Y,X}:\Hom_A(N\otimes_BY,X)\cong \Hom_B(Y,\Hom_A(N,X))$ be the adjunction isomorphisms. For $f\in\Hom_B(M\otimes_AX,Y)$ and $g\in\Hom_A(N\otimes_BY,X)$, put $\tilde{f}=\eta_{X,Y}(f)$ and $\tilde{g}=\eta'_{Y,X}(g)$.

Denote by $\Psi_X$ the composition $\xymatrix@C=2pc{N\otimes_{B} M\otimes_{A} X\ar[r]^(0.6){\psi\otimes1_X}& A\otimes_{A}X\ar[r]^(0.6){\cong} &X}$, and denote by $\Phi_Y$ the composition $\xymatrix@C=2pc{M\otimes_{A} N\otimes_{B} Y\ar[r]^(0.6){\phi\otimes1_Y}& B\otimes_{B}Y\ar[r]^(0.6){\cong} &Y}$.
Let $\delta:1_{A\text{-}\Mod}\longrightarrow \Hom_B(M,M\otimes_A-)$ and $\epsilon:M\otimes_A \Hom_B(M,-)\longrightarrow 1_{B\text{-}\Mod}$ be the unit and the counit of the adjoint pair $(M\otimes_A-,\Hom_B(M,-))$, respectively. Let $\delta ':1_{B\text{-}\Mod}\longrightarrow \Hom_A(N,N\otimes_B-)$
and $\epsilon ':N\otimes_B \Hom_A(N,-)\longrightarrow 1_{A\text{-}\Mod}$ be the unit and the counit of the adjoint pair $(N\otimes_B-,\Hom_A(N,-))$, respectively.
By \cite[Proposition 2.4]{ELG}, there is a recollement
induced by the idempotent element ${\scriptsize\left( \begin{array} {cc}
	1_A	&0  \\
0	& 0
	\end{array}\right)}$:
$$\xymatrix@!=8pc{ (B/\Im\phi)\text{-}\Mod \ar[r] |{Z_B}& {\Lambda_{\left( \phi,\psi\right) }\text{-}\Mod }\ar@<-3ex>[l]_{Q_B}
				\ar@<+3ex>[l]^{P_B} \ar[r]|{U_A} \ar[r]&  A\text{-}\Mod
				\ar@<-3ex>[l]_{T_A} \ar@<+3ex>[l]^{H_A}},$$ where
$Q_B$ is given by $\left( \begin{array} {c}
	X\\
Y
	\end{array}\right)_{f,g}\mapsto\  {_{B/\Im\phi}(B/\Im\phi)}\otimes_B \Coker f$;
$Z_B$ is given by ${_{B/\Im\phi}}Y\mapsto \left( \begin{array} {c}
	0\\
_BY
	\end{array}\right)_{0,0}$;
$P_B$ is given by $\left( \begin{array} {c}
	X\\
Y
	\end{array}\right)_{f,g}\mapsto {_{B/\Im\phi}}(B/\Im\phi)\otimes_B \Ker \tilde{g}$;
$T_A$ is given by $X\mapsto \left( \begin{array} {c}
	X\\
M\otimes_A X
	\end{array}\right)_{1,\Psi_X}$; $U_A$ is given by $\left( \begin{array} {c}
	X\\
Y
	\end{array}\right)_{f,g}\mapsto X$; $H_A$ is given by $X\mapsto\left( \begin{array} {c}
	X\\
\Hom_A(N,X)
	\end{array}\right)_{\widetilde{\Psi_X},\epsilon '_X}$ with $\widetilde{\Psi_X}=\Hom_A(N,\Psi_X)\circ\delta '_{M\otimes_A X}$.
Moreover, there is another recollement induced by the idempotent element ${\scriptsize\left( \begin{array} {cc}
	0	&0  \\
0	& 1_B
	\end{array}\right)}$:
$$\xymatrix@!=8pc{  (A/\Im\psi)\text{-}\Mod \ar[r] |{Z_A}&  {\Lambda_{\left( \phi,\psi\right) }\text{-}\Mod }\ar@<-3ex>[l]_{Q_A}
				\ar@<+3ex>[l]^{P_A} \ar[r]|{U_B} \ar[r]& B\text{-}\Mod
				\ar@<-3ex>[l]_{T_B} \ar@<+3ex>[l]^{H_B}}, $$
where $Q_A$ is given by $\left( \begin{array} {c}
	X\\
Y
	\end{array}\right)_{f,g}\mapsto \ {_{A/\Im\psi}(A/\Im\psi)}\otimes_A\Coker g$;
$Z_A$ is given by ${_{A/\Im\psi}}X\mapsto \left( \begin{array} {c}
	_AX\\
0
	\end{array}\right)_{0,0}$; $P_A$ is given by $\left( \begin{array} {c}
	X\\
Y
	\end{array}\right)_{f,g}\mapsto \ {_{A/\Im\psi}(A/\Im\psi)}\otimes_A\Ker \tilde{f}$;
$T_B$ is given by $Y\mapsto \left( \begin{array} {c}
	N\otimes_BY\\
Y
	\end{array}\right)_{\Phi_Y,1}$; $U_B$ is given by $\left( \begin{array} {c}
	X\\
Y
	\end{array}\right)_{f,g}\mapsto Y$;
$H_B$ is given by $Y\mapsto \left( \begin{array} {c}
	\Hom_B(M,Y)\\
Y
	\end{array}\right)_{\epsilon_Y,\widetilde{\Phi_Y}}$ with $\widetilde{\Phi_Y}=\Hom_B(M,\Phi_Y) \circ\delta_{N\otimes_BY}$.

To verify whether the above two recollements satisfy the condition (P), we require the following lemma.
Let $A$ be a ring and $e\in A$ be an idempotent element. Then there is a recollement
($(A/AeA)$-$\Mod$, $A$-$\Mod$, $eAe$-$\Mod$, $i^\ast,i_\ast,i^!,j_!,j^\ast,j_\ast)$
induced by $e$, where $i^*=A/AeA\otimes _A-$, $i_*=A/AeA\otimes _{A/AeA}-$, $i^!=\Hom _A(A/AeA, -)$,
$j_!=Ae\otimes_{eAe}-$, $j^*=eA\otimes _A-$ and $j_*=\Hom _{eAe}(eA, -)$, see \cite[Example 2.7]{Psa14}.

\begin{lemma}\label{idem-recollement} The recollement induced by $e$ satisfies the condition (P) if and only if the canonical map
$Ae\otimes_{eAe}eA\rightarrow A$ is a monomorphism.
\end{lemma}
\begin{proof}
The counit $\varepsilon _A: j_!j^*A\rightarrow A$ is just the canonical map
$Ae\otimes_{eAe}eA\rightarrow A$. Therefore, if the recollement
satisfies the condition (P), then the canonical map
$Ae\otimes_{eAe}eA\rightarrow A$ is a monomorphism.
Conversely,
if the canonical map
$Ae\otimes_{eAe}eA\rightarrow A$ is a monomorphism, then $\varepsilon _A$ is a monomorphism.
Since the left adjoint functors $j_!$ and $j^*$ preserve coproducts,
the counit $\varepsilon : j_!j^*\rightarrow 1_{\mathcal{A}}$ also preserves coproducts, that is,
$\varepsilon _{ \amalg A}=\amalg \varepsilon _A$.
Since $A$-$\Mod$ is a Grothendieck category, we get that $\varepsilon _{ \amalg A}$ is a monomorphism,
and then $\varepsilon _{P}$ is a monomorphism
for any projective module $P$.
\end{proof}

Now let's return to the recollements arising from Morita ring.

\begin{lemma}\label{4.2}
  Let $\Lambda=\Lambda_{(\phi,\psi)}:={\scriptsize\left( \begin{array} {cc}
	A	&_AN_B  \\
_BM_A	& B
	\end{array}\right)  }$ be a Morita ring.

{\rm(1)} The recollement ($(B/\Im\phi)$-$\Mod$, $\Lambda$-$\Mod$, $A$-$\Mod$, $Q_B, Z_B, P_B, T_A, U_A,$
\linebreak
$H_A)$
satisfies the condition (P) if and only if
$\phi$ is a monomorphism.

{\rm(2)} The recollement ($(A/\Im\psi)$-$\Mod$, $\Lambda$-$\Mod$, $B$-$\Mod$, $Q_A, Z_A, P_A, T_B, U_B,$
\linebreak
$H_B)$
satisfies the condition (P) if and only if
$\psi$ is a monomorphism.

\end{lemma}
\begin{proof}
Let $e={\scriptsize\left( \begin{array} {cc}
	1_A	&0  \\
0	& 0
	\end{array}\right)}$.
Then it follows from the proof of \cite[Proposition 4.1]{GP17} that the canonical map
$\Lambda e\otimes_{e\Lambda e}e\Lambda\rightarrow \Lambda$ is a monomorphism if and only if
$\phi$ is a monomorphism. Therefore, the statement (1) follows from Lemma~\ref{idem-recollement},
and (2) is proved similarly.
\end{proof}

Now we will construct cotorsion pairs on Morita ring $\Lambda=\Lambda_{(\phi,\psi)}:={\scriptsize\left( \begin{array} {cc}
	A	&_AN_B  \\
_BM_A	& B
	\end{array}\right)  }$ by the recollement ($(B/\Im\phi)$-$\Mod$, $\Lambda$-$\Mod$, $A$-$\Mod$, $Q_B, Z_B, P_B,$
\linebreak
$ T_A, U_A,
H_A)$. Let ${\mathcal{U^{'}}}$ and $\mathcal{V{'}}$ be two classes of left $B/\Im\phi$-modules,
and let ${\mathcal{U}^{''}}$ and ${\mathcal{V}^{''}}$ be two classes of left $A$-modules.
We set

 \begin{center} $\mathcal{M}^{\mathcal{U'}}_{\mathcal{U''}}=\{\left(\begin{array}{cc}
X\\
Y
\end{array} \right)_{f,g}|\ B/\Im\phi\otimes_B\Coker f\in{{\mathcal{{U'}}} },X\in{\mathcal{U''}}, f$ is a  monomorphism$\},$
\end{center}

 \begin{center}$\mathcal{N}^{\mathcal{V'}}_{\mathcal{V''}}=\{\left(\begin{array}{cc}
X\\
Y
\end{array} \right)_{f,g}|\ B/\Im\phi\otimes_B\Ker \tilde{g}\in{\mathcal{V'}},X\in{\mathcal{V''}
},\tilde{g}$ is an epimorphism$\}$.
\end{center}

\begin{corollary}\label{cor-motia-contor-1} Let $({\mathcal{U^{'}}},\mathcal{V{'}})$ and $ ({\mathcal{U}^{''}},{\mathcal{V}^{''}})$ be cotorsion pairs in $B/\Im\phi$-$\Mod$ and $A$-$\Mod$, respectively. Assume that $\phi$ is a monomorphism.

{\rm (1)} If $\Tor^{A}_1(M,\mathcal{U''})=0$ then
 $(^{\bot}{\mathcal{N}_{\mathcal{V^{''}}}^{\mathcal{V^{'}}}},
\mathcal{N}_{\mathcal{V^{''}}}^{\mathcal{V^{'}}})$ is a cotorsion pair in $\Lambda$-$\Mod$;
If $\Ext^{1}_A(N,\mathcal{V''})=0$ then
$(\mathcal{M}_{\mathcal{U^{''}}}^{\mathcal{U^{'}}},
({\mathcal{M}_{\mathcal{U^{''}}}^{\mathcal{U^{'}}}})^\bot)$
is a cotorsion pair in $\Lambda$-$\Mod$.

{\rm (2)} If $\Tor^{A}_1(M,\mathcal{U''})=0$
and $\Ext^{1}_A(N,\mathcal{V''})=0$, then $(\mathcal{M}_{\mathcal{U^{''}}}^{\mathcal{U^{'}}},
\mathcal{N}_{\mathcal{V^{''}}}^{\mathcal{V^{'}}})$
is a cotorsion pair in $\Lambda$-$\Mod$; and it is complete and hereditary
if so are $({\mathcal{U^{'}}},\mathcal{V{'}})$ and $ ({\mathcal{U}^{''}},{\mathcal{V}^{''}})$.
\end{corollary}
\begin{proof}
Since $\phi$ is a monomorphism,  it follows from Lemmas~\ref{4.2} and ~\ref{3.4}
that $\mathbb{L}_1Q_B(\left(\begin{array}{cc}
X\\
Y
\end{array} \right)_{f,g})=0$ if and only if $f$ is a monomorphism,
and $\mathbb{R}_1P_B(\left(\begin{array}{cc}
X\\
Y
\end{array} \right)_{f,g})$
\linebreak
$=0$ precisely when $\widetilde{g}$ is a monomorphism.
So, the statement follows from Theorem~\ref{3.3}, Proposition~\ref{prop-two-contor-equal} and Theorem~\ref{3.8(1)}.
\end{proof}

Similarly, consider the recollement ($(A/\Im\psi)$-$\Mod$, $\Lambda$-$\Mod$, $B$-$\Mod$, $Q_A,$
\linebreak
$ Z_A, P_A, T_B, U_B, H_B)$. Let ${\mathcal{C{'}}}$ and $\mathcal{D{'}}$ be two classes of left $A/\Im\psi$-modules,
and let ${\mathcal{C}^{''}}$ and ${\mathcal{D}^{''}}$ be two classes of left $B$-modules.
We set
 \begin{center}
 $\mathcal{M}^{\mathcal{C'}}_{\mathcal{C''}}=\{\left(\begin{array}{cc}
X\\
Y
\end{array} \right)_{f,g}|\ A/\Im\psi\otimes_A\Coker g\in{\mathcal{C'}},Y\in\mathcal{C''},
g$ is a monomorphism$\}$,
 \end{center}

 \begin{center}$\mathcal{N}^{\mathcal{D'}}_{\mathcal{D''}}=\{\left(\begin{array}{cc}
X\\
Y
\end{array} \right)_{f,g}|\  A/\Im\psi\otimes_A\Ker \tilde{f}
\in\mathcal{D'}, Y\in\mathcal{D''},
\tilde{f}$
is an epimorphism$\}$.
 \end{center}
Applying Theorem~\ref{3.3}, Proposition~\ref{prop-two-contor-equal} and
Theorem~\ref{3.8(1)}, we have the following corollary.

\begin{corollary}\label{cor-motia-contor-2} Let $({\mathcal{C^{'}}},\mathcal{D{'}})$ and $ ({\mathcal{C}^{''}},{\mathcal{D}^{''}})$ be cotorsion pairs in $A/\Im\psi$-$\Mod$ and $B$-$\Mod$, respectively. Assume that $\psi$ is a monomorphism.

{\rm (1)} If $\Tor^{B}_1(N,\mathcal{C''})=0$ then
 $(^{\bot}{\mathcal{N}_{\mathcal{D^{''}}}^{\mathcal{D^{'}}}},
\mathcal{N}_{\mathcal{D^{''}}}^{\mathcal{D^{'}}})$ is a cotorsion pair in $\Lambda$-$\Mod$;
If $\Ext^{1}_B(M,\mathcal{D''})=0$ then
$(\mathcal{M}_{\mathcal{C^{''}}}^{\mathcal{C^{'}}},
({\mathcal{M}_{\mathcal{C^{''}}}^{\mathcal{C^{'}}}})^\bot)$
is a cotorsion pair in $\Lambda$-$\Mod$.

{\rm (2)} If $\Tor^{B}_1(N,\mathcal{C''})=0$
and $\Ext^{1}_B(M,\mathcal{D''})=0$, then $(\mathcal{M}_{\mathcal{C^{''}}}^{\mathcal{C^{'}}},
\mathcal{N}_{\mathcal{D^{''}}}^{\mathcal{D^{'}}})$
is a cotorsion pair in $\Lambda$-$\Mod$; and it is complete and hereditary
if so are $({\mathcal{C^{'}}},\mathcal{D{'}})$ and $ ({\mathcal{C}^{''}},{\mathcal{D}^{''}})$.
\end{corollary}

Next, we consider the special case in which $M\otimes _AN=0$ or $N\otimes _BM=0$.
As in \cite{ZP}, we assume that $(\mathcal{U},\mathcal{X})$ and $(\mathcal{V},\mathcal{Y})$ are cotorsion pairs in $A$-$\Mod$ and $B$-$\Mod$, respectively.
We define
 \begin{center}
$\mathcal{M}^{\mathcal{V}}_{\mathcal{U}}=\{\left(\begin{array}{cc}
X\\
Y
\end{array} \right)_{f,g}|\ \Coker f\in{{\mathcal{V}} },X\in{\mathcal{U}}, f$ is a  monomorphism$\},$
 \end{center}

 \begin{center} $\mathcal{N}^{\mathcal{Y}}_{\mathcal{X}}=\{\left(\begin{array}{cc}
X\\
Y
\end{array} \right)_{f,g}|\ \Ker \tilde{g}\in{\mathcal{Y}},X\in{\mathcal{X}
},\tilde{g}$ is an epimorphism$\}$,
 \end{center}

 \begin{center} $\mathcal{M}^{\mathcal{U}}_{\mathcal{V}}=\{\left(\begin{array}{cc}
X\\
Y
\end{array} \right)_{f,g}| \ \Coker g\in{\mathcal{U}},Y\in\mathcal{V},
g$ is a monomorphism$\}$,
 \end{center}

 \begin{center} $\mathcal{N}^{\mathcal{X}}_{\mathcal{Y}}=\{\left(\begin{array}{cc}
X\\
Y
\end{array} \right)_{f,g}| \ \Ker \tilde{f}
\in\mathcal{X}, Y\in\mathcal{Y},
\tilde{f}$
is an epimorphism$\}$.
 \end{center}
\noindent Applying Corollary~\ref{cor-motia-contor-1} and ~\ref{cor-motia-contor-2}, we get the following two corollaries.

\begin{corollary}\label{cor-motia-contor-3}
Assume that $M\otimes _AN=0$.

{\rm (1)} If $\Tor^{A}_1(M,\mathcal{U})=0$ then
 $(^{\bot}\mathcal{N}^{\mathcal{Y}}_{\mathcal{X}}, \mathcal{N}^{\mathcal{Y}}_{\mathcal{X}})$ is a cotorsion pair in $\Lambda$-$\Mod$;
If $\Ext^{1}_A(N,\mathcal{X})=0$ then
$(\mathcal{M}^{\mathcal{V}}_{\mathcal{U}},
(\mathcal{M}^{\mathcal{V}}_{\mathcal{U}})^\bot)$
is a cotorsion pair in $\Lambda$-$\Mod$.

{\rm (2)} If $\Tor^{A}_1(M,\mathcal{U})=0$
and $\Ext^{1}_A(N,\mathcal{X})=0$, then $(\mathcal{M}^{\mathcal{V}}_{\mathcal{U}},
\mathcal{N}^{\mathcal{Y}}_{\mathcal{X}})$
is a cotorsion pair in $\Lambda$-$\Mod$; and it is complete and hereditary
if so are $(\mathcal{U},\mathcal{X})$ and $(\mathcal{V},\mathcal{Y})$.

\end{corollary}

\begin{corollary}\label{cor-motia-contor-4}
Assume that $N\otimes _BM=0$.

{\rm (1)} If $\Tor^{B}_1(N,\mathcal{V})=0$ then
 $(^{\bot}\mathcal{N}^{\mathcal{X}}_{\mathcal{Y}}, \mathcal{N}^{\mathcal{X}}_{\mathcal{Y}})$ is a cotorsion pair in $\Lambda$-$\Mod$;
If $\Ext^{1}_B(M,\mathcal{Y})=0$ then
$(\mathcal{M}^{\mathcal{U}}_{\mathcal{V}},
(\mathcal{M}^{\mathcal{U}}_{\mathcal{V}})^\bot)$
is a cotorsion pair in $\Lambda$-$\Mod$.

{\rm (2)} If $\Tor^{B}_1(N,\mathcal{V})=0$
and $\Ext^{1}_B(M,\mathcal{Y})=0$, then $(\mathcal{M}^{\mathcal{U}}_{\mathcal{V}},
\mathcal{N}^{\mathcal{X}}_{\mathcal{Y}})$
is a cotorsion pair in $\Lambda$-$\Mod$; and it is complete and hereditary
if so are $(\mathcal{U},\mathcal{X})$ and $(\mathcal{V},\mathcal{Y})$.
\end{corollary}

\begin{remark}\label{rek-zhu-Mao-cotor}
{\rm (1) When $M=0$ or $N=0$, parts of Corollary~\ref{cor-motia-contor-3} and ~\ref{cor-motia-contor-4} were obtained
by Zhu et al.~\cite[Theorem 3.4 and Proposition 3.7]{ZR}.

(2) Applying Theorems~\ref{3.6}, we improve the results of
\cite[Theorem 4.4]{ML}. Indeed, let $T={\scriptsize\left( \begin{array} {cc}
	A	&0  \\
U	& B
	\end{array}\right)  }$ as in \cite{ML}, where $U$ an $B$-$A$-bimodule.
Consider the first recollement
($B$-$\Mod$, $\Lambda$-$\Mod$, $A$-$\Mod$, $Q_B, Z_B, P_B, T_A, U_A, H_A)$ with
$\mathcal{C}_1$ and $\mathcal{C}_2$ (resp. $\mathcal{D}_1$ and $\mathcal{D}_2$) being
classes of left $A$-modules (resp. $B$-modules). Then we have $\mathcal{M}_{\mathcal{C}_1}^{\mathcal{D}_1}=\mathfrak{B}_{\mathcal{D}_1}^{\mathcal{C}_1}$ and
$\mathcal{N}_{\mathcal{C}_2}^{\mathcal{D}_2}=\mathfrak{A}_{\mathcal{D}_2}^{\mathcal{C}_2}$,
where the symbols $\mathfrak{B}_{\mathcal{D}_1}^{\mathcal{C}_1}$ and $\mathfrak{A}_{\mathcal{D}_2}^{\mathcal{C}_2}$
are adopted from \cite{ML}. Applying Theorem~\ref{3.6} to this recollement yields that $(\mathcal{C}_1, \mathcal{C}_2)$ and $(\mathcal{D}_1, \mathcal{D}_2)$ are cotorsion pairs
if and only if $(\mathfrak{B}_{\mathcal{D}_1}^{\mathcal{C}_1}, \mathfrak{A}_{\mathcal{D}_2}^{\mathcal{C}_2})$ is a cotorsion pair,
requiring only $\Tor _1^A(U, ^\bot{\mathcal{C}_2})=0$ due to the exactness of $H_A$.
Similarly, applying Theorem~\ref{3.6} to the second recollement
($A$-$\Mod$, $\Lambda$-$\Mod$, $B$-$\Mod$, $Q_A, Z_A, P_A, T_B, U_B, H_B)$,
we obtain that $(\mathcal{C}_1, \mathcal{C}_2)$ and $(\mathcal{D}_1, \mathcal{D}_2)$ are cotorsion pairs
if and only if $(\mathfrak{A}_{\mathcal{D}_1}^{\mathcal{C}_1}, \mathfrak{I}_{\mathcal{D}_2}^{\mathcal{C}_2})$ is a cotorsion pair,
under the assumption that $\Ext _B^1(U, {\mathcal{D}_1^\bot})=0$ since $T_B$ is exact.

(3) Applying Theorem~\ref{3.8(1)} and Proposition~\ref{3.16} to these recollements,
we improve the result in \cite[Theorem 5.6]{ML}. Indeed, from the first recollement we get that if $\Tor _1^A(U, {\mathcal{C}_1})=0$, and
$(\mathcal{C}_1, \mathcal{C}_2)$ and $(\mathcal{D}_1, \mathcal{D}_2)$ are complete hereditary cotorsion pairs, then
$(\mathfrak{B}_{\mathcal{D}_1}^{\mathcal{C}_1}, \mathfrak{A}_{\mathcal{D}_2}^{\mathcal{C}_2})$ is a complete hereditary cotorsion pair,
and the converse holds in case $U\otimes _A\mathcal{C}_1\subseteq \mathcal{D}_1$ and $\Tor _1^A(U, ^\perp{\mathcal{C}_2})=0$.
Similarly, from the second recollement we get that if $\Ext ^1_B(U, {\mathcal{D}_2})=0$, and
$(\mathcal{C}_1, \mathcal{C}_2)$ and $(\mathcal{D}_1, \mathcal{D}_2)$ are complete hereditary cotorsion pairs, then
$(\mathfrak{A}_{\mathcal{D}_1}^{\mathcal{C}_1}, \mathfrak{I}_{\mathcal{D}_2}^{\mathcal{C}_2})$ is a complete hereditary cotorsion pair,
and the converse holds in case $\Hom _B(U, \mathcal{D}_2)\subseteq \mathcal{C}_2$ and $\Ext ^1_B(U, {\mathcal{D}_1^\perp})=0$.
  }
\end{remark}

In \cite{ZP}, Zhang-Cui-Rong construct four cotorsion pairs on Morita rings under the assumption
that $\phi = 0 = \psi$ or $M\otimes _AN=0=N\otimes _BM$.
It is worth noting that their construction differs from ours, particularly because
the rings $A$, $B$ and $\Lambda$ are more closely linked in the case $\phi = 0 = \psi$.
However, in Corollaries~\ref{cor-motia-contor-1} and~\ref{cor-motia-contor-2}, we consider a different setting where either
$\phi$ or $\psi$ is a monomorphism. Therefore, the approach in our construction differs from that in \cite{ZP}.
The following example, due to Zhang-Cui-Rong \cite{ZP}, shows that the cotorsion pairs constructed here differ
from those in \cite{ZP}.

\begin{example}\label{ex}{\rm
Let $A=B$ be the path algebra $k(1 \rightarrow 2)$ where k is a field. We write the conjunction of paths from right to left. Thus $e_1 A e_2=0$ and $e_2 A e_1 \cong k$. Take $M=N=A e_2 \otimes_k e_1 A$. Then $M \otimes_A N=0=N \otimes_A M$. Let $\Lambda$ be the Morita ring ${\scriptsize\left( \begin{array} {cc}
	A	&N \\
N	& A
	\end{array}\right) } $.
In this case, $_AM= {_{A}N}$ is a projective left $A$-module and $M_A=N_A$ is a projective right $A$-module.
Take $({\mathcal{U}},\mathcal{X})=(\mathcal{V},\mathcal{Y})=(_A{\mathcal{P}},A$-$\Mod)$ and set
 \begin{center}
$\mathcal{M}_1=\{\left(\begin{array}{cc}
X\\
Y
\end{array} \right)_{f,g}|\ \Coker f\in{_A\mathcal{P}},X\in {_A\mathcal{P}},
f$ is a monomorphism$\}$,
 \end{center}

 \begin{center}$\mathcal{M}_2=\{\left(\begin{array}{cc}
X\\
Y
\end{array} \right)_{f,g}|\ \Coker g\in{_A\mathcal{P}},Y\in {_A\mathcal{P}},
g$ is a monomorphism$\}$,
 \end{center}

 \begin{center} $\mathcal{N}_1=\{\left(\begin{array}{cc}
X\\
Y
\end{array} \right)_{f,g}|\
\tilde{g}$
is an epimorphism$\}$,
 \end{center}

 \begin{center} $\mathcal{N}_2=\{\left(\begin{array}{cc}
X\\
Y
\end{array} \right)_{f,g}|\
\tilde{f}$
is an epimorphism$\}$.
 \end{center}
\noindent Then it follows from Corollary~\ref{cor-motia-contor-3} and~\ref{cor-motia-contor-4} that both $(\mathcal{M}_1, \mathcal{N}_1)$
and $(\mathcal{M}_2, \mathcal{N}_2)$ are complete hereditary cotorsion pairs in $\Lambda$-$\Mod$.
We remark that the results in \cite{HJS} cannot be used to glue together $({\mathcal{U}},\mathcal{X})$ and $(\mathcal{V},\mathcal{Y})$, since
the functor $i^!$ in this recollement fails to be exact.
Moreover, the cotorsion pairs constructed here are distinct from those of Cui et al.~\cite{ZP}; see \cite[Example 4.3]{ZP}.
}
\end{example}

\medskip

\noindent {\footnotesize {\bf ACKNOWLEDGMENT.}
This work is supported by
the National Natural Science Foundation of China (12561008),
the project of Young and Middle-aged Academic and Technological leader of Yunnan
(202305AC160005) and the Basic Research Program of Yunnan Province (202301AT070070).}


\begin{thebibliography}{99}
\bibitem{AKL11} L. Angeleri H\"{u}gel, S. Koenig and Q. H. Liu, Recollements and tilting objects, J. Pure
Appl. Algebra. 215 (2011), 420--438.

\bibitem{BBD82} A. A. Be\u ilinson, J. Bernstein and P. Deligne, Faisceaux pervers, in Analysis and topology on
	singular spaces, I (Luminy, 1981), 5--171, Ast\'erisque, 100, Soc. Math. France, Paris, 1982.

\bibitem{CW} W. Cao, J. Wei and K. Wu, Recollements and $n$-cotorsion pairs, arXiv:2403.04220.

\bibitem{Chen13} J. M. Chen, Cotorsion pairs in a recollement of triangualted categories, Comm. Algebra. 41 (2013),
no. 8, 2903--2915.

\bibitem{Chen12} X. W. Chen, A recollement of vector bundles, Bull. Lond. Math. Soc. 44 (2012), 271--284.

\bibitem{ZP} J. Cui, S. Rong and P. Zhang, Cotorsion pairs and model structures on Morita rings, J. Algebra 661 (2025),  1--81.

 \bibitem{EEE} E. E. Enochs and O. M. G. Jenda, Relative Homological Algebra, de Gruyter Exposit. Math., vol. 30, Walter De Gruyter, Berlin, New York, 2000.
\bibitem{FZ17} J. Feng and P. Zhang, Types of serre subcategories of Grothendieck categories, J. Algebra 508 (2017), 16--34.
\bibitem{FV} V. Franjou and T. Pirashvili, Comparison of abelian categories recollements, Doc. Math. 9 (2004), 41--56.

\bibitem{FH} X. R. Fu and Y. G. Hu, The recollements of abelian categories: cotorsion dimensions and cotorsion triples, Bull. Iranian. Math. Soc. 48 (2022),
no. 3, 963--977.
\bibitem{GKP21}  N. Gao, S. K\"{o}nig and C. Psaroudakis, Ladders of recollements of abelian categories, J. Algebra. 579 (2021), 256--302.
\bibitem{GP17} N. Gao and C. Psaroudakis, Gorenstein homological aspects of monomorphism categories via Morita rings, Algebr. Represent.
Theory 20 (2017), 487--529.

\bibitem{JR} J. R. Garcia-Rozas, Covers and envelopes in the category of complexes of Module, Reserch Notes in Math., vol.407, Chapman \& Hall/CRC, Boca Raton, FL, 1999.
\bibitem{GJ} J. Gillespie, Cotorsion pairs and degreewise homological model structures, Homol. Homotopy Appl. 10 (2008), no. 1, 283--304.
\bibitem{GR} R. G\"{o}bel and J. Trlifaj, Approximations and Endomorphism Algebras of Modules, GEM 41. Berlin-New York: Walter de Gruyter. (2006).
\bibitem{ELG} E. L. Green and C. Psarounddkis, On Artin algebras arising from Morita contexts, Algebr. Represent. Theory 17 (2014), no. 5, 1485--1525.

\bibitem{HJ} J. He and J. He, n-cotorision pairs and recollements of extriangulated categories, J. Algebra Appl. (2025), accepted online.

\bibitem{HM} M. Hovey, Cotorsion pairs, model category structures and representation theory, Math. Z. 241 (2002), no. 3, 553--592.
\bibitem{HJS} J. S. Hu, H. Y. Zhu and R. M. Zhu, Gluing and lifting exact model structures for the recollement of exact categories, Bull. Malays. Math. Sci. Soc. 48 (2025), no. 5, 173.
\bibitem{LVY14} Q. H. Liu, J. Vit\'{o}ria and D. Yang, Gluing silting objects, J. Nagoya Math. 216 (2014), 117--151.
\bibitem{MX2} X. Ma and Z. Y. Huang, Torsion pairs in recollements of abelian categories, Front. Math. China. 13 (2017), no. 4,
875--892.
\bibitem{MZ25} X. Ma and P. Y. Zhou, Dimensions and cotorsion pairs in recollements of extriangulated categories, Comm.
Algebra 53 (2025), no. 1, 242--258.

\bibitem{MX} X. Ma and P. Y. Zhou, n-cotorision pairs in a recollement of extriangulated categories, arXiv:2508.20331.

\bibitem{ML} L. X. Mao, Cotorsion pairs and approximation classes over formal triangular matrix rings, J. Pure Appl. Algebra. 224 (2020), no. 6, 106271.


\bibitem{Psa14} C. Psaroudakis, Homological theory of recollements of abelian categories, J. Algebra. 398 (2014),
	63--110.
\bibitem{PV14} C. Psaroudakis and J. Vit\'{o}ria, Recollements of module categories, Appl. Categ. Struct. 22 (2014), no. 4, 579--593.


\bibitem{SL} L. Salce, Cotorsion theories for abelian groups, Symposia Math. 23 (1979), 11--32.

\bibitem{TO79} H. Tachikawa, K. Ohtake, Colocalization and localization in Abelian categories, J. Algebra 56 (1979),
1--23.

\bibitem{WWZ22} L. Wang, Q. J. Wei and  H. C. Zhang, Recollements of extriangulated categories, Colloq. Math. 167 (2022),
239--259.

\bibitem{YQ26} J. R. Yang and Y. Y. Qin, ICE-closed subcategories and epibricks over recollements, Algebra Colloq., accepted.

\bibitem{ZC17} C. Zhang, On the global cohomological width of artin algebras, Colloq.
Math., 146 (2017), 31--46.

\bibitem{Z17} Y. Y. Zhang, Reduction of wide subcategories and recollements,  Algebra Colloq. 30 (2023), no. 4, 713--720.

\bibitem{ZR} R. M. Zhu, Y. Y. Peng and N. Q. Ding, Recollements associated to cotorsion pairs over upper triangular matrix rings, Publ. Math. Debrecen
98 (2021), no. 1, 83--113.
 \end{thebibliography}
\end{document}